\documentclass{amsart}
\usepackage{amsmath,amsthm,amssymb,xypic}
\usepackage[all]{xy}
\usepackage[top=1in, bottom=1in, left=1in, right=1in]{geometry}
\usepackage{color}
\usepackage[colorlinks=true,linkcolor=blue,urlcolor=blue,citecolor=blue]{hyperref}
\usepackage[draft,footnote,nomargin,inline]{fixme}
\usepackage{enumerate}
\usepackage {paralist}

\makeatletter
\@namedef{subjclassname@2010}{%
 \textup{2010} Mathematics Subject Classification}
\makeatother

\title{Cremona linearizations of some classical varieties}

\author[C.~Ciliberto, M.~A.~Cueto, M.~Mella, K.~Ranestad, P.~Zwiernik]{Ciro Ciliberto\and Maria Angelica Cueto\and Massimiliano Mella\and Kristian Ranestad\and Piotr Zwiernik}
\address{Universit\`a degli Studi di Roma Tor Vergata, Via della
  Ricerca Scientifica 1, 00133, Rome,
  Italy.}\email{cilibert@axp.mat.uniroma2.it}

\address{
Mathematics Department, 
Columbia University,  MC 4403, 
2990 Broadway,
New York, NY 10027, USA
}
\email{macueto@math.columbia.edu}
\address{Dipartimento di Matematica e Informatica\\
Universit\`a di Ferrara\\
Via Machiavelli 35\\
44100 Ferrara, Italia} \email{mll@unife.it}
\address{University of Oslo, PB 1053 Blindern, 0316, Oslo, Norway}\email{ranestad@math.uio.no}
\address{University of California Berkeley, Department of Statistics,
CA 94720-3860, Berkeley, USA }\email{pzwiernik@berkeley.edu}
\date{\today}

\theoremstyle{plain}

\newtheorem{theorem}{Theorem}[section]
\newtheorem{proposition}[theorem]{Proposition}
\newtheorem{lemma}[theorem]{Lemma}

\theoremstyle{definition}

\newtheorem{definition}[theorem]{Definition}
\newtheorem*{claim}{Claim}

\newtheorem{exampleth}[theorem]{Example}
\newenvironment{example}{\begin{exampleth}}{\hfill $\diamond$\\ \end{exampleth}}

\theoremstyle{remark}

\newtheorem{remark}[theorem]{Remark}

\DeclareMathOperator{\rk}{rk}

 \DeclareMathOperator{\Sec}{Sec}

\newcommand{\QED}{\ifhmode\unskip\nobreak\fi\quad {\rm Q.E.D.}} 

\newcommand{\f}{\varphi}

\newcommand{\C}{\mathbb{C}}

\renewcommand{\H}{\mathcal{H}}

\renewcommand{\L}{\mathcal{L}}

\newcommand{\Af}{\mathbb{A}}

\renewcommand{\P}{\mathbb{P}}

\newcommand{\rat}{\dasharrow}

\newcommand{\Segre} {{\rm Seg}}

\begin{document}

\maketitle

\begin{abstract} 
  In this paper we present an effective method for {linearizing}
  rational varieties of codimension at least two under Cremona
  transformations, starting from a given parametrization.  Using these
  linearizing Cremonas, we simplify the equations of secant and
  tangential varieties of some classical examples, including Veronese,
  Segre and Grassmann varieties.  We end the paper by treating the
  special case of the Segre embedding of the $n$--fold product of
  projective spaces, where cumulant Cremonas, arising from algebraic
  statistics, appear as specific cases of our general construction.
\end{abstract}

 \section*{Introduction}\label{sec:intro}
 Computations in Algebraic Geometry may be very sensitive to the choice of
 coordinates. Often, by picking the appropriate coordinate
 system, calculations or expressions can be greatly simplified.
 Change of variables with rational functions are classically known as
\emph {Cremona transformations} and give a huge flexibility when
 dealing with systems of polynomial equations.   In this paper, we
 focus on varieties and maps defined over $\C$.
 
 Cremona transformations, one of the most venerable topics in algebraic geometry,
were widely studied in the second half of XIX and the first of XX century. 
They became again fashionable in recent times, after the spectacular developments of birational 
 geometry due to Mori, Kawamata, Koll\'ar, et al. and even more recently by Birkar, Cascini, Hacon, McKernan (see \cite {dr} and references therein). However, despite this great progress, Cremona transformations still reserve a great deal of surprises. The aim of the present, mainly expository, paper is to show how useful they can be in studying some 
 classical geometric invariants of complex projective varieties, linking previous independent work of the last three authors (see \cite {GHPRS, MP, pwz-2010-cumulants}). For a recent survey on the properties of the group of Cremona transformations, we refer the reader to~\cite{SegreCantat}.
 
The work of Mella and Polastri \cite{MP} shows that any rational variety $X$ of
codimension at least two in $\P^ r$ can be \emph{linearized} by a Cremona
transformation: a \emph{linearizing Cremona map} is a transformation that maps $X$
birationally to a linear subspace of $\P^ r$. In
Section~\ref{sec:crem-equiv} we provide a proof of the
aforementioned result, close to the original one in spirit, but best suited for effective computations.

 If  $X$ in $\P^ r$ admits a birational linear projection to a linear
subspace, a Cremona linearizing map can be directly constructed as a \emph{triangular Cremona
transformation} (see \S~\ref{sec:trianguar}). 
In fact, in Section~\ref{sec:constr-crem-transf} we present a
systematic approach to Cremona transformations linearizing a
rational variety. They turn out to be building blocks for
the \emph{cumulant Cremona transformations}, which we present in
Section~\ref{sec:cum}.  In Section \ref {sec:crem-repr} we discuss the effect of linearizing Cremona
transformations on tangential and secant varieties.  We devote Section~\ref{sec:examples} to the  study of a number of classical examples, including Veronese, Segre and Grassmann varieties, and in these cases we observe an interesting feature of linearizing Cremonas: they tend to simplify also tangential and secant varieties. 

The final Section~\ref{sec:cum} is, as we said, devoted to  {cumulant Cremona transformations} which appear already, in a simple case, in Section~\ref{sec:examples}. Cumulants arise from algebraic statistics (see \cite {pwz-2010-cumulants}) and can be viewed as the choice of preferable coordinates in which  varieties coming from algebraic
statistics simplify. For instance, Sturmfels and Zwiernik~\cite{pwz-2011-bincum} used
cumulants to simplify the equations of the Segre embedding of the
$n$-fold product of $\P^1$ and of its tangent variety. More recent results in the same direction are contained in \cite{MM,MOZ}.  Cumulants are
particular instances of  linearizing Cremona transformations. 
We conclude by indicating how to 
generalize some combinatorial formulas in \cite{MOZ, pwz-2011-bincum,pwz-2010-cumulants}. 


\section{Construction of some Cremona transformations}
\label{sec:constr-crem-transf}

In this section we present some recipes for constructing 
Cremona transformations. We focus on two specific closely-related types:
\emph{monoidal extensions} and \emph{triangular Cremona transformations}. All constructions in this paper may be
seen as iterated applications of monoidal extensions 
as in 
Section~\ref{sec:mono-extens-rati} (see also \cite {CS}).

\subsection{Basics on Cremona transformations} A \emph{Cremona transformation} is a birational map
 \begin{equation}
   \label{eq:1}
 \varphi\colon \P^ r\dasharrow \P^r,  \qquad
 [x_0,\ldots,x_r]\mapsto [F_0(x_0,\ldots,x_r),\ldots, 
 F_r(x_0,\ldots,x_r)],
\end{equation}
where $F_i(x_0,\ldots,x_r)$ are coprime homogeneous
polynomials of the same degree $\delta>0$, for $0\leqslant i\leqslant r$. The
inverse map is also a Cremona transformation, and it is defined by
coprime homogeneous polynomials $G_i(x_0,\ldots,x_r)$ of degree $\delta'>0$, for $0\leqslant i\leqslant r$. In this case, we say that
$\varphi$ is a \emph{$(\delta, \delta')$--Cremona transformation}.  The subscheme 
${\rm Ind}(\varphi):=\{F_i(x_0,\ldots,x_r)=0\}_{0\leqslant i\leqslant r}$ is the {\em indeterminacy locus} of $\varphi$.  Since the
composition of $\varphi$ and its inverse is the identity, we have
\[G_i(F_0(x_0,\ldots,x_r), \ldots, F_r(x_0,\ldots,x_r))=\Phi\cdot x_i,
\,\,\, \text {for}\,\,\, 0\leqslant i\leqslant r\] where $\Phi$ is a
homogeneous polynomial of degree $\delta\cdot \delta'-1$. The
hypersurface ${\rm Fund}(\varphi):=\{\Phi=0\}$ is the
\emph{fundamental locus} of $\varphi$ and its support is the
\emph{reduced fundamental locus} ${\rm Fund}_{\rm red}(\varphi)$. In
general one cannot reconstruct ${\rm Fund}(\varphi)$ from ${\rm
  Fund}_{\rm red}(\varphi)$, except when ${\rm Fund}_{\rm
  red}(\varphi)$ is irreducible. Indeed, in this case, we recover
${\rm Fund}(\varphi)$ from its multiplicity value $(\delta\cdot
\delta'-1)/e$, where $e$ is the degree of ${\rm Fund}_{\rm
  red}(\varphi)$.
By construction, ${\rm Ind}(\varphi)\subset {\rm Fund}(\varphi)$ and $\phi$ is one-to-one on the complement of ${\rm Fund}_{\rm red}(\varphi)$.

Often times in this paper, the loci ${\rm Fund}_{\rm red}(\varphi)$
and ${\rm Fund}_{\rm red}(\varphi^ {-1})$ induced by the Cremona
transformation $\varphi$ are the same
hyperplane. 
In those cases, we can see $\varphi$ as a polynomial automorphism
$\varphi_a\colon \Af^ r\to \Af^ r$ (often denoted by $\varphi$)
whose extension to $ \P^ r$ contains the hyperplane at infinity as its
fundamental locus.

\subsection{Monoids} Let $X\subset \P^ r$ be a hypersurface of degree
$d$ where $r\geqslant 2$. We say that $X$
is a \emph{monoid} with \emph{vertex} $p\in \P^ r$ if $X$ is irreducible and $p$ is a point in $X$
of multiplicity exactly $d-1$. Note that a monoid can have more than one
vertex.  If we choose projective coordinates in such a way that
$p=[0,\ldots,0,1]$, then the defining equation of $X$ is
\begin{equation*}\label{eq:monoid}
  f(x_0,\ldots,x_{r})=f_{d-1}(x_0,\ldots,x_{r-1})\,x_{r}+
  f_{d}(x_0,\ldots,x_{r-1})=0,\end{equation*} 
where $f_{d-1}$ and $f_d$  are homogeneous polynomials of degree $d-1$ and $d$ respectively and $f_{d-1}$ is nonzero. The hypersurface $X$ is irreducible if and only if  $f_{d-1}$ and $f_d$ are coprime.

A monoid $X$ is rational. Indeed, the
projection of $X$ from its vertex $p$ onto a hyperplane $H$ not
passing through $p$ is a birational map $\pi\colon X\dasharrow H\cong \P^
{r-1}$. If $H$ has equation $x_r=0$, then the inverse map
$\pi^ {-1}\colon \P^ {r-1} \dasharrow X $ is given by 
\[
[x_0,\ldots, x_{r-1}] 
\mapsto [{f_{d-1}(x_0,\ldots,x_{r-1})}x_0,\ldots, {f_{d-1}(x_0,\ldots,x_{r-1})}x_{r-1}, -f_{d}(x_0,\ldots,x_{r-1})].
\]
The map $\pi$ is called the \emph{stereographic projection} of $X$ from
$p$.  Its indeterminacy locus  is $p$. Each line through $p$ contained in $X$ gets contracted to a point under $\pi$.  The set of all such lines is defined by the
equations $\{f_d=f_{d-1}=0\}$. This is the indeterminacy locus
of $\pi^ {-1}$, whereas the hypersurface of $H$ with equation
$\{x_r=f_{d-1}=0\}$ is contracted to $p$ by the map $\pi^ {-1}$.

\subsection{Monoidal extensions of rational maps}
\label{sec:mono-extens-rati}
Let $\omega\colon \P^ r\dasharrow \P^ r$ be a dominant rational map
defined, in homogeneous coordinates, as in~\eqref{eq:1}. 
The homogeneous polynomials $F_0, \ldots, F_r$ have the same degree
$\delta>0$ and are coprime. We construct a \emph{monoidal extension} $\Omega$ of  $\omega$ as
follows. First, we embed $\P^ r$ in $\P^ {r+1}$ as the hyperplane
 $H=\{x_{r+1}=0\}$ and we consider the point $p=[0,\ldots, 0, 1] \in \P^
{r+1}$. 
Fix an integer $d\ge \delta$, a nonzero polynomial
$h(x_0,\ldots,x_r)$ of degree $d-\delta$ and an irreducible {monoid} of degree
$d$ with {vertex} at $p$ defined by
\[
  f(x_0,\ldots,x_{r+1})=f_{d-1}(x_0,\ldots,x_{r})\,x_{r+1}+
  f_{d}(x_0,\ldots,x_{r})=0.\]
Then we let $\Omega\colon \P^
{r+1}\dasharrow \P^ {r+1}$ be defined by
\begin{equation*}
[x_0,\ldots, x_{r+1}] \mapsto  [h(x_0,\ldots,x_r) F_0(x_0,\ldots,x_r),\ldots, h (x_0,\ldots,x_r)  F_r(x_0,\ldots,x_r), f(x_0,\ldots,x_{r+1})].\label{eq:3}
\end{equation*}
Note that $p$ is an indeterminacy point of $\Omega$. If $\pi: \P^
{r+1}\dasharrow \P^ r$ is the projection from $p$ to $H$, we have
\begin{equation}\label{eq:proj}\pi\circ \Omega =\omega\circ \pi. \end{equation}

\begin{lemma}\label{lem:tranf} The map $\Omega\colon \P^
  {r+1}\dasharrow \P^ {r+1}$ is dominant and has the same degree as
  $\omega$. Hence $\Omega$ is a Cremona transformation if and
  only if $\omega$ is.
\end{lemma}

\begin{proof}
  By definition, the degree of the map $\Omega$ coincides with the
  degree of the induced field extension. Let
  $y=[y_0,\ldots,y_{r+1}]\in \P^ {r+1}$ be a general point and let
  $y'=[y_0,\ldots,y_{r}]=\pi(y)$.  The rational map $\Omega$ may be
  written as $\Omega(y)=[F_0(y'):\ldots:F_r(y'):
  (f_{d-1}(y')y_{r+1}+f_d(y'))/h(y')]$. In particular, $\Omega(\C(y_0,\ldots, y_{r}))=\omega(\C(y_0,\ldots, y_r))\subset \C(y_0,\ldots, y_r)$, while $\Omega(y)_{r+1}$ is linear in $y_{r+1}$ over $\Omega(\C(y_0,\ldots, y_r))$. The field extension has degree $[\C(y_0,\ldots, y_{r+1}):\Omega(\C(y_0,\ldots, y_{r}))]=[\C(y_0,\ldots, y_r):\omega(\C(y_0, \ldots, y_{r}))]$ and the lemma follows.
\end{proof}

The  indeterminacy locus of $\Omega$ (as a scheme) is the union of the cone over the locus of 
indeterminacy of $\omega$ with vertex $p$ intersected with the monoid $\{f=0\}$ and the
 codimension two  subvariety $\{h=f=0\}$.
 The reduced fundamental locus of $\Omega$ is the  union of the hypersurface
 $\{h=0\}$ and  the cone over the fundamental locus of $\omega$ with
 vertex $p$.

\subsection{Triangular Cremona transformations}\label{sec:trianguar}
\label{sec:triang-crem-transf} \emph{Triangular Cremona transformations} are obtained as
iterated applications of monoidal extensions, as we now explain.
Consider a rational map $\tau: \mathbb P^ r\dasharrow \mathbb P^ r$
defined, in affine coordinates over $\{x_0\neq 0\}$, by formulas of the type
\[ (x_1,\ldots,x_r)\to (f_1(x_1),\dots, f_i(x_1,\ldots, x_i), \ldots
f_r(x_1,\ldots,x_r))\] where each $f_i(x_1,\ldots, x_i)\in
K(x_1,\ldots,x_{i-1})[x_i]$ is nonconstant and linear in $x_i$, for
$1\leqslant i\leqslant r$.  If $f_i(x_1,\ldots, x_i)\in
K[x_1,\ldots, x_i]$ for all $1\leqslant i\leqslant r$, the
indeterminacy locus of $\tau$ is contained in the \emph{hyperplane at
  infinity} $\{x_0=0\}$.  Any such map $\tau$ is birational,  with
inverse of the same type.  To find $\tau^ {-1}$, we have
to solve the system
\[y_i= f_i(x_1,\ldots, x_i), \quad 1\leqslant i\leqslant r\]
in the indeterminates $x_1,\ldots, x_r$. This can be done stepwise as
follows. From the linear expression  $y_1=f_1(x_1)$ we find a
linear polynomial $g_1$ such that $x_1=g_1(y_1)$. Given $i>1$, assume
 we know that
\begin{equation}\label{eq:iteration}
x_h=g_h(y_1,\ldots,y_h), \quad {\rm for}\quad  
1\leqslant h< i\leqslant n
\end{equation}
where $g_h(y_1,\ldots,y_h)\in K(y_1,\ldots,y_{h-1})[y_h]$ is linear in
$y_h$.  From $y_i= f_i(x_1,\ldots, x_i)$ we obtain the expression
$x_i=\xi(x_1,\ldots,x_{i-1},y_i)$, where
$\xi(x_1,\ldots,x_{i-1},y_i)\in K(x_1,\ldots,x_{i-1})[y_i]$ is linear
in $y_i$. Substituting $x_h$ from  \eqref
{eq:iteration}, we conclude that $x_i=g_i(y_1,\ldots,y_i)$ with
$g_i(y_1,\ldots,y_i)\in K(y_1,\ldots,y_{i-1})[y_i]$ of degree 1 in
$y_i$.

\begin{example} Fix an integer $n\ge 2$ and use the same notation as
  above.  The following Cremona quadratic transformation $\omega_n$ of $\P^
  {{n+1}\choose 2}$ is triangular
\[[x_0, \ldots , x_i,\ldots, x_{ij},\ldots]
\to [x_0^ 2,\ldots , x_0 x_i,\ldots, x_0
x_{ij}-x_ix_j,\ldots], \,\,\, \text{where}\,\,\, 1\le i<j\le n.\]
The inverse is 
\[[y_0, \ldots , y_i,\ldots, y_{ij},\ldots]
\to [y_0^ 2,\ldots , y_0 y_i,\ldots, y_0
y_{ij}+y_iy_j,\ldots]\]

From a geometric viewpoint, $\omega_n$ is defined by a linear system
$\mathcal L$ of quadrics as follows.  Consider the coordinate
hyperplane $\Pi=\{ x_0=0\}$. Let $S$
be the linear subspace of $\Pi$ with equations
$\{x_0=x_1=\ldots=x_n=0\}$ and let $S'$ be the the
complementary subspace in $\Pi$ with equations
$\{x_0=x_{12}=\ldots=x_{ij}=\ldots=x_{(n-1)n}=0\}$, where $1\leqslant
i<j\leqslant n$.  Then $\mathcal L$ cuts out  the complete linear system of
quadric cones on $\Pi$ that are singular along $S$ and that pass through the $n$
independent points $p_i\in S'$, where $p_i$ is the torus-fixed point with all coordinates $0$ but $x_i$, with $1\leqslant i\leqslant n$. After splitting off $\Pi$ from $\mathcal L$, the residual system
consists of all hyperplanes containing $S$.
\end{example}

\section{Cremona equivalence}
\label{sec:crem-equiv}

An irreducible variety $X\subset \P^ r$ is \emph{Cremona linearizable} (CL) if there is a \emph{linearizing Cremona transformation} of $\P^r$ which maps $X$ birationally onto a linear subspace.   It is a consequence of 
Theorem \ref {thm:mp} (presented in~\cite {MP}) that, if $X$ is rational of dimension $n\leqslant r-2$ in $\P^ r$, then $X$ is CL.
In this section we recall this theorem and present a
slightly different proof. 

\subsection{Monoids and Cremona transformations}\label{ssec:monoids}
Let $X\subset \P^ r$ be a monoid of degree $d$.  
Let $p_1, p_2\in X$ be two vertices, let 
$H_1,H_2$ be hyperplanes with $p_i\not \in H_i$, and  { consider the
stereographic projections of $X$ from $p_i$, which is the restriction of the projection
$\pi_i\colon \P ^r\dasharrow H_i$ from $p_i$, with
$i=1,2$.} The map
\[
\pi_{X,p_1,p_2}:=\pi_2\circ \pi_1^ {-1}\colon H_1\dasharrow H_2
\] is a Cremona transformation. If $p_1=p_2=p$, then
$\pi_{X,p}:=\pi_{X,p,p}$ does not depend on $X$, being the (linear)
\emph{perspective} of $H_1$ to $H_2$ with center $p$. From now on, we
restrict to the case when $p_1\neq p_2$. In this setting, the map
$\pi_{X,p_1,p_2}$ depends on $X$ and is in general nonlinear.  In the
following we assume that $H_1$ and $H_2$ have equations $x_r=0$ and
$x_{r-1}=0$ respectively, and $p_1=[0,\ldots,0,1], p_2=[0,\ldots,
0,1,0]$. The defining equation of $X$ has the form
\begin{equation*}\label{eq:bimon}
f_d(x_0,\ldots, x_{r-2})+x_{r-1}g_{d-1}(x_0,\ldots,
x_{r-2})+x_rh_{d-1}(x_0,\ldots, x_{r-2})+ x_{r}x_{r-1}f_{d-2}(x_0,
\ldots, x_{r-2})=0.
\end{equation*}
Then 
\[
\pi_{X,p_1,p_2}( [x_0,\ldots, x_{r-1}] )=  [(f_{d-2}x_{r-1}+h_{d-1})x_0,\ldots,
(f_{d-2}x_{r-1}+h_{d-1})x_{r-2},-f_d-x_{r-1}g_{d-1}].
\]


{

\begin{lemma}  \label{lem:mp16_1} Let $Z\subset \P^ r$, with $r\geqslant 3$, be an irreducible variety of  dimension $r-2$ and let $p\in \P^ r$ be  such that the projection of $Z$ from $p$ is birational to its image. For $d\gg 0$ there is a monoid of degree $d$ with vertex $p$, containing $Z$ but not containing the cone $C_p(Z)$
  over $Z$ with vertex $p$.
\end{lemma}

\begin{proof}    Let $V\to \P^ r$ be the  blow--up of $\P^ r$ at $p$. We denote by $E$ the exceptional divisor and by $H$ the proper transform of a general hyperplane of $\P^ r$ and by $Z'$ the proper transform of $Z$. 

Consider $M_d =\vert dH-(d-1)E\vert= \vert (d-1)(H-E)+H\vert$,
i.e. the proper transform on $V$  of the  linear system of monoids we
are interested in. We have
\begin{equation}\label{eq:mon}
\dim (M_d)=\frac{2d^{r-1}}{(r-1)!}+\frac{(r-1)d^{r-2}}{(r-2)!}+O(d^{r-3}).
\end{equation}
Since the projection of $Z$ from $p$ is birational, the line bundle $\mathcal O_{Z'}(H-E)$ is big and nef. Then 
by \cite [Theorem~1.4.40, Vol.\ I, p.\ 69]{laz} it follows that
\begin{equation*}\label{eq:av}
h^ 0(Z',\mathcal O_{Z'}(d(H-E)))= \frac  {\delta} {(r-2)!} d^ {r-2}+ O(d^ {r-3}), \,\,\, \text {for}\,\,\, d\gg 0
\end{equation*} 
where $\delta=(H-E)^ {r-2}\cdot Z'>0$ is the degree of the variety obtained under the projection of $Z$ from $p$, i.e. the degree of  the cone $C_p(Z)$. Thus, if $M'_d$ is the sublinear system of $M_d$ of the divisors containing $Z'$, then 
$$\dim (M'_d)\geqslant \dim(M_d)- h^ 0(Z',\mathcal O_{Z'}(d(H-E)))=
 \frac{2d^{r-1}}{(r-1)!}+\frac{(r-1-\delta)d^{r-2}}{(r-2)!}+O(d^{r-3}). $$

 We let $M''_d$ be the sublinear system of $M'_d$ of the divisors
 containing the proper transform $Y$ of the cone $C_p(Z)$, which is a
 hypersurface of degree $\delta$, i.e. $Y\in \vert \delta
 (H-E)\vert$. Hence, $M''_d\subseteq \vert (d-\delta-1)(H-E)+H\vert$ so
 by \eqref {eq:mon} we have
$$\dim (M''_d)\leq  \frac{2(d-\delta)^{r-1}}{(r-1)!}+\frac{(r-1)(d-\delta)^{r-2}}{(r-2)!}+O(d^{r-3})=
\frac{2d^{r-1}}{(r-1)!}+\frac{(r-1-2\delta)d^{r-2}}{(r-2)!}+O(d^{r-3}).
$$
Hence 
$$\dim(M'_d)-\dim(M''_d)=\frac{\delta d^{r-2}}{(r-2)!}+O(d^{r-3})>0, \quad \text{for}\quad d\gg 0,$$
as we wanted to show.\end{proof}

\begin{lemma} \label{lem:mp16} Let $Z\subset \P^ r$ be an irreducible
  variety of positive dimension $n\leqslant r-3$ and let $p_1,p_2\in
  \P^ r$ be distinct points such that the projection of $Z$ from the
  line $\ell$ joining $p_1$ and $p_2$ is birational to its image. For
  $d\gg 0$ there is a monoid of degree $d$ with vertices $p_1$ and
  $p_2$, containing $Z$ but not containing any of the cones $C_{i}(Z)$
  over $Z$ with vertices $p_i$, for $i=1,2$.
\end{lemma}

\noindent\emph{Proof.} We start with the following

\begin{claim} It suffices to prove the assertion for $n=r-3$.
\end{claim} 

\begin{proof} [Proof of the Claim] Consider the projection of $\P^ r$
  to $\P^ {n+3}$ from a general linear subspace $\Pi$ of dimension
  $r-n-4$ 
and call $Z', p'_1,p'_2, \ell'$ the projections of $Z,p_1,p_2, \ell$
respectively. Then $Z'$ is birational to $Z$ and it is still true that
the projection of $Z'$ form $\ell'$ is birational to its image. The
dimension of $Z'$ is $n-3$.

Assume   the assertion holds for $Z', p'_1,p'_2$ and let $F'\subset \P^ {n+3}$ be a monoid of degree $d\gg 0$ with vertices
$p'_1,p'_2$ containing $Z'$ but not  $C_{i}(Z')$, for $i=1,2$. Let
$F\subset \P^ r$ be the cone over $F'$ with vertex $\Pi$. Then $F$ is
a monoid with vertices $p_1,p_2$ containing $Z$. It does not contain
either one of $C_{i}(Z)$, for $i=1,2$, otherwise $F'$ would contain
one of the cones $C_{i}(Z')$, for $i=1,2$,  contradicting our
hypothesis on $F'$.\end{proof}

We can thus assume from now on that $n=r-3$. Fix two hyperplanes $H_1$
and $H_2$, where $p_1\notin H_1$ and $p_2\notin H_2$.  Let $Z_1$ and
$Z_2$ be the projections from $p_1$ and $p_2$ to $H_1$ and $H_2$,
respectively.
We set $p'_{3-i}:=\pi_i(p_{3-i})$, for $i=1,2$. Our result follows by
Lemma~\ref {lem:mp16_1} and the following claim:

\begin{claim} It suffices to prove that for $d\gg 0$, there is a monoid of degree $d$ in $H_i$ with
  vertex $p'_i$ containing $Z_i$ but not containing the cone $C(Z_i)$
  over $Z_i$ with vertex $p'_i$, for $i=1,2$.
\end{claim}

\begin{proof} [Proof of the Claim] Let $F'_i\subset H_i$ be such a
  monoid, and let $F_i$ be the cone over $F'_i$ with vertex $p_{3-i}$,
  for $i=1,2$.  Then $F_i$ is a monoid with vertex $p_{i}$ (by
  construction, we can take any point in the line joining $p_i$ and
  $p_{3-i}$ as its vertex). 
In addition, $F_i$  contains
  $Z$ but does not contain $C_i(Z)$ (same argument as in the proof of
  the previous Claim). Then the assertion of Lemma~\ref{lem:mp16_1}
  holds for a general linear combination $F$ of $F_1$ and
  $F_2$.\end{proof}
 }

 Let $Z, p_1, p_2$ be as in Lemma \ref {lem:mp16}. Fix a general monoid $X\supset Z$ with vertices in $p_1$ and $p_2$; by Lemma \ref {lem:mp16} $X$ does not contain the cones $C_i(Z)$. Let $H_1$ and $H_2$
 be hyperplanes such that $p_i\not \in H_i$, for $i=1,2$. Let $Z_i$ be
 the projection of $Z$ from $p_i$ to $H_i$, for $i=1,2$. Then the map $\varphi:=\pi_{X,p_1,p_2}:H_1\rat H_2$ is birational and $Z_1$ is not contained in the indeterminacy locus of both $\varphi$ and $\varphi^{-1}$.
Thus $\varphi$ induces a birational transformation $\phi\colon
 Z_1\dasharrow Z_2$. For future reference we summarize this construction in the following Proposition.

\begin{proposition}\label{prop:proj} In the above setting, the
  double projection  $\varphi\colon H_1\dasharrow H_2$,
 (resp.\ $\varphi^ {-1}$) is defined at the general
  point of $Z_1$ (resp.\ of $Z_2$), hence it defines a birational map
  $\phi\colon Z_1\dasharrow Z_2$ (resp.\ $\phi^{-1}\colon
  Z_2\dasharrow Z_1$).
\end{proposition}


\subsection {Cremona equivalence} \label{ssec:cremeq}
  Let $X,
Y\subset \P^ r$ be irreducible, projective varieties. We say that $X$
and $Y$ are \emph{Cremona equivalent} (CE) if there is a Cremona
transformation $\omega\colon \P^ r\dasharrow \P^ r$ such that $\omega$ [resp. $\omega^ {-1}$] is defined
at the general point of $X$ [resp. of $Y$] and such that $\omega$ maps $X$ to $Y$ [accordingly $\omega^ {-1}$ maps $Y$ to $X$]. This is an equivalence relation
among all irreducible subvarieties of $\P^r$. 

The following result is due to Mella and Polastri~\cite[Theorem
1]{MP}. We present an alternative proof, close to the original ideas,  but  more computational in spirit. 
\begin{remark}
  We take the opportunity of correcting a mistake in the proof
  of~\cite[Theorem 1]{MP}. In the notation of \cite[Theorem 1]{MP},
  let $T_i=\f_{\H_{i}}(X)$, $Y_i=T_i\cap (x_{n+1}=0)$ and $Z$ the cone
  over $T_i$ with vertex $q_1$. Let $S$ be the monoid containing $Y_i$
  and $W_i$ the projection of $T_i$ onto $S$. Then
  $X_{i+1}=\pi_{q_2}(W_i)$ and not the projection of a general
  hyperplane section of $Z$ as written in the published paper.
\end{remark}

\begin{theorem}\label{thm:mp} Let  $X, Y\subset \P^r$, with $r\geqslant 3$, be two irreducible varieties of positive dimension $n<r-1$.  Then
  $X,Y$ are CE if and only if they are birationally
  equivalent. \end{theorem}

\begin{proof} We prove the nontrivial implication. 

 {
\begin{claim} We may assume that the
  projection of $Y$ from any coordinate subspace of dimension $m$ is
  birational to its image if $r> n+m+1$ and dominant to $\P^{r-m-1}$
  if $r\le n+m+1$.

\end{claim}

\begin{proof}[Proof of the Claim] If we choose the $r+1$ torus-fixed
  points of $\P^ r$ to be generic (which we can do after a generic
  change of coordinates), then each one of the coordinate
  subspaces of a given dimension $m$ (spanned by $m+1$ coordinate
  points) is generic in the corresponding Grassmannian, hence the
  assertion follows.\end{proof} }

  Let $Z$ be a smooth variety and
  let $\phi\colon Z\dasharrow X$ and $\psi\colon Z\dasharrow Y$ be birational
  maps. Passing to affine coordinates, we may assume that $\phi$ and $\psi$
  are given by equations
\[
x_j=\phi_j(t), \text{ and }\; y_j=\psi_j(t),\,\,\, \text {for} \,\,\, 1\leqslant j\leqslant  r,
\]
where $\phi_j, \psi_j$ are rational functions on
$Z$ and $t$ varies in a suitable dense open subset of $Z$. 

We prove the theorem by constructing a sequence of birational maps $\varphi_i: Z\dasharrow X_i\subset \P^ r$, 
with $X_i$ projective varieties, for $0\leqslant i\leqslant  r$, such that:\par
\begin{inparaenum} 
\item [(a)] $\varphi_0=\phi$ and $\varphi_r=\psi$, thus $X_0=X$ and $X_r=Y$;\par
\item [(b)] for $0\leqslant i\leqslant r-1$, there is a Cremona
  transformation $\omega_i\colon\P^ r\dasharrow \P^ r$, such that
  $\omega_i$ (resp.\ $\omega_i^ {-1}$) is defined at the general point
  of $X_i$ (resp.\ of $X_{i+1}$), and it satisfies $\omega_i(X_i)=X_{i+1}$
  (accordingly, $\omega_i^ {-1}(X_{i+1})=X_i$) and
  $\varphi_{i+1}=\omega_i\circ \varphi_i$.
\end{inparaenum}

The construction is done recursively. We assume $\varphi_i$ is of the form
\[
\varphi_i(t)= (\tilde \phi_{i,1}(t),\ldots, \tilde \phi_{i,r-i}(t),\psi_{r-i+1}(t), \ldots, \psi_r(t)), \,\,\, \text{for}\,\,\, t\in Z \,\,\, \text{and}\,\,\, 0\leqslant i\leqslant  r-1,
\]
where the $\tilde \phi_{i,j}$'s are suitable rational functions on
$Z$. For $i=0$, the starting case, we fix $\tilde \phi_{0,i}= \phi_i$
for all $0\leqslant i\leqslant r$.  Thus, requirement (a) is
satisfied.

Assume $0\leqslant i\leqslant r-1$. In order to perform the step from
$i$ to $i+1$, we consider the map
\[
g_i\colon Z\dasharrow \mathbb A^ {r+1}, 
\qquad g_i(t)= (\tilde \phi_{i,1}(t),\ldots, \tilde
\phi_{i,r-i}(t),\psi_{r-i}(t),\psi_{r-i+1}(t),\ldots, \psi_r(t)).
\]
Let $Z_i$ be the closure of the image of $g_i$, and
$\pi_{j}\colon \mathbb{A}^{r+1}\to \mathbb{A}^{r}$ the projection to
the coordinate hyperplane $\{x_{j}=0\}$ from the point at infinity of
the axis $x_{j}$, for all $1\leqslant j\leqslant r+1$. 
We have $\varphi_i=\pi_{r-i+1}\circ g_i$, and since $\varphi_i$ is
birational onto its image, the same holds for $g_i$.
{
\begin{claim}\label{sub:claim1} The projection of $Z_i$ from a general
  point of the space at infinity of the affine linear space $\Pi_i:=  \{x_{r-i+1}=\ldots =x_{r+1}=0\}$ is birational to its
  image. \end{claim}

\begin{proof} [Proof of the Claim] The variety $Z_i$ is not a
  hypersurface, then by \cite [Theorem 1]{CC}, the locus of points
  from which the projection is not birational has dimension strictly
  bounded by the dimension of $Z_i$.  We may therefore assume that
  $\dim(\Pi_i)=r-i-1< n$. On the other hand the map $\psi$ is
  birational, therefore we may as well assume that $i<r-1$ .  So it
  remains to prove the result in the range $0<r-i-1<n$. 

 The projection of $Z_i$ from $\Pi_i$ is the closure of the image of the
  map
\[
h_i\colon  Z\dasharrow  \mathbb A^ {i+1}, \qquad  h_i(t)= (\psi_{r-i}(t),\psi_{r-i+1}(t),\ldots, \psi_r(t)).
\]
By the previous claim applied to $Z$, 
either $h_i$ is birational to its image (if 
$i\geqslant n$) 
or $h_i$ is dominant. In the former case the projection from a general
point of $\Pi_i$ is also birational, so the assertion follows. In the
latter case the cone over $Z_i$ with vertex $\Pi_i$ is the whole
$\P^r$, and the assertion follows from \cite [Theorem
1]{CC}.\end{proof}}

By the Claim, we can make a general change of the first $r-i$
coordinates of $g_i$ so that $\varphi_{i+1}:= \pi_{r-i+i}\circ g_i$ 
is
birational to its image $X_{i+1}$. Finally, iterated applications of
Proposition \ref{prop:proj} show that also requirement (b) is
satisfied, thus ending the proof.\end{proof}

\subsection{Linearizing Cremona}\label{ssec:lincrem}
Theorem \ref {thm:mp} ensures that any variety $X$ of dimension $n\leq
r-2$ in $\P^ r$ is CE to a hypersurface in a $\P^
{n+1}{\subset\P^r}$. If, in addition, $X$ is rational, then it is CL,
e.g.\ it is CE to the subspace $\{x_{n+1}=\ldots =x_r= 0\}$.  In
particular, suppose that there is a linear subspace $\Pi$ of dimension
$r-n-1$ of $\P^r$ such that the projection from $\Pi$ induces a
birational map $\pi\colon X\dasharrow \P^ n$.  Equivalently, $X$
admits an affine parametrization of the form
\begin{equation}\label{eq:parameter}
  x_i=t_i,\quad {\rm for}\quad 1\leqslant i\leqslant  n,\qquad
  x_j=f_j(t_1,\ldots,t_n), \quad {\rm for}\quad n+1\leqslant j\leqslant  r,\end{equation}
where the $f_i$'s are rational functions of $t_1,\ldots,t_n$.  For instance, 
smooth toric varieties and Grassmannians enjoy this property (see Section~\ref {sec:examples}). 

Using~\eqref{eq:parameter}, we define a  Cremona
map $\phi\colon \mathbb P^ r\dasharrow \mathbb P^ r$ in affine
coordinates as
\begin{equation}
  \phi(x_1,\ldots,x_r) = (x_1,\ldots, x_n, x_{n+1}-f_{n+1}(x_1,\ldots, x_n),\ldots, x_{r}-f_{r}(x_1,\ldots, x_n)).\label{eq:basiccremona}
\end{equation}
The map $\phi$ gives a birational equivalence between $X$ and the
subspace $\{x_{n+1}=\ldots =x_r= 0\}$, hence $\phi$ linearizes $X$.
The above construction can be slightly modified to make it more general.  Fix a collection of
rational functions $g_i(x_1,\ldots,x_{i-1})$ and $h_i(x_1,\ldots,x_{i-1} )$, with $n+1\leqslant i\leqslant
r$, with all the $h_i(x_1,\ldots,x_{i-1} )$'s are nonzero. 
We replace the $i$--th coordinate of $ \phi$
with the expression
\[ \phi_i(x_1,\ldots,x_r)= h_i(x_1,\ldots,x_{i-1} )\big (
x_i-f_i(x_1,\ldots,x_n)\big ) +g_i(x_1,\ldots,x_{i-1}), \qquad i=n+1,
\ldots, r .\]
 The following is clear: 
\begin{lemma}
  The image $\phi(X)$ is the linear subspace $\{x_{n+1}=\ldots =x_r=
  0\}$ if and only if the functions $g_i$ vanish on $X$ for 
  $n+1 \leqslant i\leqslant r$.
\end{lemma}

\section{Secant and tangential varieties}
\label{sec:crem-repr}

In this section, we focus on Cremona transformations of secant and
tangential varieties.  Similar techniques can be applied to
{osculating varieties}, although we will not do this here.

\begin{definition}
  Let $X\subset \P^r$ be a variety of dimension $n$. The 
\emph{$k$--secant variety} $\Sec_k(X)$ of $X$ (simply $\Sec(X)$ if $k=1$)  is the Zariski
closure of the union of all $(k+1)$--secant linear spaces of dimension
$k$ to $X$, i.e. those containing $k+1$ linearly independent  points of $X$. The
\emph{$k$--defect} of $X$ is $\min\{r, n(k+1)+k\}-\dim(\Sec_k(X))$ (which is nonnegative), and $X$ 
is \emph{$k$--defective} if  the {$k$--defect} is positive.
\end{definition}
The $k$--secant variety of $X$ has expected dimension $nk+n+k$. It is parametrized as
\begin{equation}
  \label{eq:30}
  \psi\colon {\rm Sym}^ {k+1}(X)\times \P^k \dashrightarrow \Sec_k(X)\subset \P^r,  \quad ([p^{(0)},
  \ldots, p^{(k)}], [s_0,\ldots,
  s_k])\mapsto \sum_{j=0}^k s_j p^{(j)}.
\end{equation}

Assume that there is a codimension $n+1$ linear subspace $\Pi$ such
that the projection from $\Pi$ induces a birational map $\pi\colon
X\dashrightarrow \P^n$. From Section \ref {sec:crem-equiv}, we know
that $X$ can be parametrized as in~\eqref {eq:parameter}.  Then, we can
combine the maps $\psi$ and $\pi$ to simplify the parametrization of
$\Sec_k(X)$, as we now show.

Pick affine variables $s_1,\ldots, s_k$ and set $s_0:=1-\sum_{j=1}^k
s_j$ in~\eqref{eq:30}. Consider $k+1$ vectors of unknowns
\[{\bf t}_i=(t_{i1},\ldots, t_{in}) \quad \text{ for}\quad 0\leqslant
i\leqslant k.\] 
Then, $\Sec_k(X)$ is parametrized as follows
\begin{equation*}
 \label{eq:param2} 
  x_i=\begin{cases}
    s_0t_{0i}+s_1t_{1i}+\ldots +s_kt_{ki}& \quad \text{ for}\quad 1\leqslant i\leqslant n,\\
   s_0f_i({\bf t}_0)  +s_1f_i({\bf t}_1)+\ldots+s_kf_i({\bf t}_k)& \quad
   \text{ for}\quad n+1\leqslant i \leqslant r. 
  \end{cases}
 \end{equation*}
 We let $\phi$ be the Cremona transformation from
 \eqref{eq:basiccremona}, that linearizes $X$.  Applying $\phi$ to
 $\Sec_k(X)$ gives
\begin{equation}
\label{eq:param1} 
 x_i=\begin{cases}
s_0t_{0i}+s_1t_{1i}+\ldots +s_kt_{ki}& \; \text{ for }\;
 1\leqslant i\leqslant n,\\
s_0f_i({\bf t}_0)  +s_1f_i({\bf t}_1)+\ldots+s_kf_i({\bf t}_k)-f_i(s_0{\bf
  t}_0+s_1{\bf t}_1+\ldots+s_k{\bf t}_k)& \; \text{ for } \;
n+1\leqslant i \leqslant r.
\end{cases}
\end{equation}

This change of coordinates can be useful for computing geometric
invariants of $X$, such as its $k$--defect. The next example,
illustrates this situation.
\begin{example}\label{ex:sec}
Suppose that the  $f_i$ in ~\eqref {eq:parameter} are quadratic
polynomials. In this case, $X$  is a projection of the Veronese variety, hence it is 1--defective. We write the  homogeneous decomposition of $f_i$
\[ f_i=f_{i0}+f_{i1}+f_{i2}, \,\,\, \qquad\text{for}\,\,\, n+1\leqslant i\leqslant  r,\] where $f_{ij}$
is the homogeneous component  of $f_i$ of degree $j$. Let $\Phi_i$ be
the bilinear form associated to $f_{i2}$. Then, the parametrization
\eqref{eq:param1} yields
the expression
\begin{equation*}
\label{eq:param3} 
x_i =\begin{cases}
s_0t_{0i}+s_1t_{1i}+\ldots +s_kt_{ki},& \quad \text{ for}\quad 1\leqslant i\leqslant n,\\
\sum\limits_{j=0}^ k s_j(1-s_j) f_{i2} ({\bf t}_j)-
2\sum\limits_{0\leqslant u<v\leqslant k }s_us_v\Phi_i({\bf t}_u, {\bf t}_v),& \quad
\text{ for}\quad n+1\leqslant i \leqslant r.
\end{cases}
\end{equation*}

Suppose that $k=1$. Then
$$
x_i=s_0(1-s_0)f_{i2} ({\bf t}_0)+s_1(1-s_1)f_{i2} ({\bf t}_1)-2s_0s_1\Phi_i({\bf t}_0, {\bf t}_1),\,\,\, \text{for}\,\,\, n+1\leqslant  i\leqslant r.
$$
Since $s_0=1-s_1$, then  $s_0(1-s_0)=s_1(1-s_1)=s_0s_1$ and  we obtain
\[
x_i= s_0 s_1 f_{i2} ({\bf t}_1-{\bf t}_0)\qquad\mbox{for}\quad n+1\leqslant i \leqslant r.
\]
Replacing ${\bf t}_0-{\bf t}_1$ with ${\bf u}:=(u_{1},\ldots,u_{n})
$ 
 and setting ${\bf t}_0=:{\bf t}=(t_1,\ldots, t_n)$ and $s_1=:s$ yields
\begin{equation*}
\label{eq:param2b} 
x_i =\begin{cases}
t_{i}+su_{i},& \quad \text{ for}\quad 1\leqslant i\leqslant n,\\
s(1-s)f_{i2} ({\bf u}),& \quad
\text{ for}\quad n+1\leqslant i \leqslant r.
\end{cases}
\end{equation*}
The image of a general secant line is a conic with two points in the linear image of the variety $X$.
The dimension of the secant variety can be deduced from  the rank of the Jacobian of this parametrization. \end{example}

Next we discuss the interplay between tangential varieties and Cremona
transformations.
\begin{definition}
    Let $X\subset \P^r$ be a variety. The 
\emph{tangential variety} $T(X)$ of $X$ is  the Zariski closure of the
union of all tangent spaces to $X$ at smooth points of $X$. 
\end{definition}
Assume that $X$ has dimension $n$. The tangential variety has expected
dimension $2n$. If $X$ is (locally) parametrized by a map
\[ {\bf t}=(t_1,\ldots, t_n)\in U\mapsto [x_0({\bf t}), \ldots,
x_r({\bf t})]\in X,\]
where $U\subset \C^ n$ is a suitable nonempty open subset, then $T(X)$ is represented by
\begin{equation*}
  \label{eq:4}
  \tau\colon U\times \C^{n}  \dashrightarrow T(X), \quad
({\bf t},{\bf s})=(t_1,\ldots, t_n, s_1,\ldots, s_n) \mapsto
  [x_0({\bf t})+\sum\limits_{j=1}^n s_j \frac {\partial x_0}{\partial
  t_j}({\bf t}), \ldots, x_r({\bf t})+\sum\limits_{j=1}^n s_j \frac {\partial x_r}{\partial
  t_j}({\bf t})].
\end{equation*}
Assume again that $X$ is described as in \eqref
{eq:parameter}. Then, the parametric equations of $T(X)$ have a
simplified expression
\begin{equation*}
\label{eq:param3a} 
x_i =\begin{cases}
t_{i}+s_i,& \quad \text{ for } \quad 1\leqslant i\leqslant  n,\\
f_i({\bf t})+\sum\limits_{j=1}^ ns_j\frac {\partial f_i}{\partial
  t_j} ({\bf t}),& \quad \text{ for }\quad n+1\leqslant i\leqslant   r.
\end{cases}
\end{equation*}
Under the linearizing Cremona transformation $\phi$ from Section~\ref{sec:crem-equiv}, the variety $T(X)$
has image
\begin{equation}
\label{eq:param4} 
x_i =\begin{cases}
t_{i}+s_i,& \;\text{ for }\; 1\leqslant i\leqslant  n,\\
f_i({\bf t})-f_i({\bf t}+{\bf s})
+\sum\limits_{j=1}^ ns_j\frac {\partial f_i}{\partial t_j} ({\bf t}),& \; \text{ for }\; n+1\leqslant i\leqslant   r.
\end{cases}
\end{equation}

\begin{example}\label{ex:tan}
Assume the $f_i$'s in~\eqref
{eq:parameter} are homogeneous quadratic polynomials. Then \eqref{eq:param4} becomes
\begin{equation}
\label{eq:param5} 
x_i =\begin{cases}
t_{i}+s_i,& \quad \text{ for}\quad 1\leqslant i\leqslant  n,\\
-f_i({\bf s}),& \quad \text{ for}\quad n+1\leqslant i\leqslant   r.
\end{cases}
\end{equation}
Formula~\eqref{eq:param5} describes a cone with vertex the space at
infinity of the $n$--dimensional linear space $\{x_{n+1}=\ldots =
x_r=0\}$, over the variety parametrically represented by the last
$n-r$ coordinates of~\eqref{eq:param5}
\begin{equation*}
\label{eq:param6} 
x_i=-f_i({\bf s}),  \qquad  \text{ for}
\quad n+1\leqslant i\leqslant   r.
\end{equation*}
\vspace{-5ex}
\end{example}

In Section \ref {sec:examples} we will see how the equations of secant
and tangential varieties simplify in classical defective cases, as
predicted by the above example.  If the parametrization involves forms
of degree higher than $2$, the tangent variety is in general no longer
transformed to a cone.  In Section \ref{sec:cum} we will see
alternative linearizing Cremonas that work better for certain
varieties. For instance, for Segre varieties cumulant Cremonas enable
us to write the tangential variety in the form (\ref{eq:param5}) even
though the parametrizing polynomials are not quadratic.

\section{ Cremona linearization of some classical varieties}
\label{sec:examples}

Segre, Veronese and Grassmannian varieties and their secants
play a key role in the study of determinantal
varieties.  
 Here we describe some triangular Cremona transformations that
 linearize these varieties, and we will compute the image of their
 secant varieties under these transformation. Similar considerations
 can be applied to \emph{Spinor varieties} (see \cite  {an} for a
 parametrization of these varieties), and to \emph{Lagrangian Grassmannians} $LG(n,2n)$, etc., on which we do not dwell here.

\subsection{Segre varieties} \label{sec:segre-varieties} The
\emph{Segre variety} $\Segre (r_1,\ldots, r_k)$ is  the image of 
$\P^ {r_1}\times \ldots\times \P^ {r_k}$ under the Segre embedding in $\P^ r$, with $r+1=\prod_{i=1}^ k (r_i+1)$ (we may assume $r_1\geqslant r_2\geqslant \ldots\geqslant r_k\geqslant 1$).  Sometimes we may use the exponential notation $\Segre (m^ {h_1}_1,\ldots, m^ {h_k}_k)$ if  $m_i$ is repeated $h_i$ times, for $1\leqslant i\leqslant k$.

In this section, we find Cremona linearizations for $\Segre(m,n)$ and
we show how they simplify the equations for their secant varieties. In
Section~\ref{sec:cum} we will extend this to higher Segre varieties.

We interpret $ \P^{mn+m+n}$ as the space of nonzero $(m+1)\times
(n+1)$ matrices modulo multiplication by a nonzero scalar, so we have
coordinates $[x_{ij}]_{0\leqslant i\leqslant n, 0\leqslant j\leqslant
  m}$ in $ \P^{mn+m+n}$.  Then, $\Segre(m,n)$ is defined by the rank condition
\[\rk(x_{ij})_{0\leqslant i\leqslant n,  0\leqslant j\leqslant m}=1.\]
This condition amounts to equate to zero all $2\times 2$ minors of the
matrix ${\bf x}=(x_{ij})_{0\leqslant i\leqslant n, 0\leqslant
  j\leqslant m}$. We pass to affine coordinates by setting $x_{00}=1$,
and we let
\[
x=\begin{pmatrix} 1&x_{01}&x_{02}&\cdots & x_{0n}\\
x_{10}&x_{11}&x_{12}&\cdots & x_{1n}\\
\vdots&&& & \vdots\\
x_{m0}&x_{m1}&x_{m2}&\cdots & x_{mn}
  \end{pmatrix}
\]
be the corresponding matrix. Then the affine equations of
$\Segre(m,n)$ are $\{ x_{ij}-x_{i0}x_{0j}=0\}_{1\leqslant i\leqslant
  n, 1\leqslant j\leqslant m}$. This shows that $\Segre(m,n)$ has
parametric equations of type \eqref {eq:parameter} with parameters
$x_{i0}, x_{0j}$, for $1\leqslant i\leqslant n, 1\leqslant j\leqslant
m$.

As in Section~\ref {ssec:lincrem} a linearizing
affine Cremona has equations (in vector form)
    \begin{equation}
  ( y_{ij})_{0\leq i\leq m, 0\leq j\leq n, (i,j)\neq (0,0)}=
     (x_{i0}, x_{0j}, x_{ij}-x_{i0}x_{0j})_{1\leqslant i\leqslant n,  1\leqslant j\leqslant m},
\label{eq:SegreNN}
   \end{equation}
   which is of type $(2,2)$ and in homogeneous coordinates reads
     \[
    [{\bf y}]=[y_{ij}]_{0\leq i\leq m, 0\leq j\leq n)}=
     [x_{00}^ 2, x_{00}x_{i0},x_{00}x_{0j}, x_{00}x_{ij}-x_{i0}x_{0j}]_{1\leqslant i\leqslant n,  1\leqslant j\leqslant m}.
     \]
The indeterminacy locus has equations $\{x_{00}=x_{i0}x_{0j}=0\}_{1\leqslant i\leqslant n,  1\leqslant j\leqslant m}$ and the reduced fundamental locus $\{x_{00}=0\}$.  

To see the image of  the secant varieties, we 
perform column operations on $x$ and use ~\eqref{eq:SegreNN} to see that 
   \[
\operatorname{rank}(x)=\operatorname{rank}\,\begin{pmatrix} 1&0&0&\cdots & 0\\
y_{10}&y_{11}&y_{12}&\cdots & y_{1n}\\
\vdots&&& & \vdots\\
y_{m0}&y_{m1}&y_{m2}&\cdots & y_{mn}
  \end{pmatrix}
   =1+\operatorname{rank} \begin{pmatrix} 
y_{11}&y_{12}&\cdots & y_{1n}\\
\vdots&& & \vdots\\
y_{m1}&y_{m2}&\cdots & y_{mn}
  \end{pmatrix}. 
   \]  
 Therefore the $k$--secant  variety to $\Segre(m,n)$ is mapped to the cone over the $(k-1)$--secant variety of $\Segre(m-1, n-1)$ with vertex along the linear image of 
   $\Segre(m,n)$.

\begin{example}\label{ex:seg}  The (first) secant and tangent variety to $\Segre(2,2)$ is the cubic 
    hypersurface defined by the $3\times 3$-determinant 
    \[
   \det ({\bf x})= x_{00}(x_{11}x_{22}-x_{12}x_{21})-x_{01}(x_{10}x_{22}-x_{20}x_{12})+x_{02}(x_{10}x_{21}-x_{20}x_{11})=0.
    \]
    In the new coordinates this hypersurface has the simpler binomial equation 
$y_{11}y_{22}- y_{12}y_{21}=0$.\end{example}

   \subsection{Projectivized tangent bundles}  The projectivized tangent bundle $TP^n$ over $\P^n$ is 
      embedded in $\Segre(n,n)$ as the traceless nonzero  $(n+1)\times(n+1)$--matrices modulo
      multiplication by nonzero scalar, i.e. as the hyperplane section ${\rm tr} ({\bf x})=0$
      of  $\Segre(n,n)$ in $ \P^{n^2+2n}$. On the affine chart $x_{00}\neq 0$, 
      we view $TP^n$ as the set of rank $1$ matrices of the form
\[
x=\begin{pmatrix} 1&x_{01}&x_{02}&\cdots & x_{0n}\\
x_{10}&x_{11}&x_{12}&\cdots & x_{1n}\\
\vdots&&& & \vdots\\
x_{n0}&x_{n1}&x_{n2}&\cdots & -x_{11}-\ldots -x_{n-1,n-1}-1
  \end{pmatrix}.
\] We can
      parametrize $TP^n$ with the $2n-1$ coordinates $x_{0i}\neq 0$, with $1\leqslant i\leqslant n$,  and $x_{ii}$, with 
      $1\leqslant i\leqslant n-1$. The parametric equations for the remaining coordinates are
      \[\left\{
\begin{array}{lr}
x_{i0}=\frac {x_{ii}}{x_{0i}} & \text{ for  } 1\leqslant i \leqslant n-1,\\
x_{n0}=-\frac {1+x_{11}+\ldots +x_{n-1,n-1}}{x_{0n}},\\
x_{ij}=\frac {x_{ii} x_{0j}} {x_{0i}}=x_{i0}x_{0j}&\text{ for  }1\leqslant i<j \leqslant n. \\
\end{array}\right.
\]

     According to Section \ref {ssec:lincrem} we have a linearizing Cremona map
 $\phi \colon\P^{n^2+2n-1}\dasharrow \P^{n^2+2n-1}$
   given in affine coordinates by
    \begin{equation}\label{eq:mat}
\begin{array}{lr}
y_{0i}=x_{0i} & \text{ if }  \,\,\,0\leqslant i \leqslant n \\
y_{ii}=x_{ii} &  \text{ if }\,\,\,1\leqslant i \leqslant n-1\\
y_{i0}=x_{ii}- x_{i0}x_{0i}  & \text{ for  } 1\leqslant i \leqslant n-1\\
y_{n0}=-({1+x_{11}+\ldots +x_{n-1,n-1}}+x_{n0}x_{0n})\\
y_{ij}=x_{ij}-x_{i0}{x_{0j}}&\text{ for  }1\leqslant i<j \leqslant n. \\
\end{array}
\end{equation}

   Performing row operations on $x$ and using \eqref {eq:mat}, we see that $x$ has rank $k$ if and only if  
    \[
y'=\begin{pmatrix} y_{10}&y_{12}&y_{13}&\cdots & y_{1n}\\
y_{21}& y_{20}&y_{23}&\cdots &y_{2n}\\
\vdots&&& & \vdots\\
y_{n1}&y_{n2}&y_{n3}&\cdots &  y_{n0}
  \end{pmatrix}
\]
has rank $k-1$. This shows that 
the $k$--th secant variety of  $TP^n$ is mapped to a cone over the  $(k-1)$-st  secant variety of   $\Segre(n-1,n-1)$ with vertex along the linear image of $TP^n$.

\begin{example} \label{ex:tp} The first secant variety of $TP^2$
  coincides with the  tangent variety and it is the cubic 
    hypersurface defined in $\P ^7$, with coordinates $[x_{ij}]_{0\leqslant i\leqslant j\leqslant 2, (i,j)\neq (2,2)}$, by the equation
        \[
  \det (x)= x_{00}^2 x_{11}+x_{00}(x_{11}^2+x_{12}x_{21}-x_{01}x_{10})-x_{01}(x_{10}x_{11}+x_{20}x_{12})-x_{02}(x_{10}x_{21}-x_{20}x_{11})=0.
    \]
    In the new coordinates it has the simpler equation 
    $y_{12}y_{21}=y_{10}y_{20}$, which defines the cone over $\Segre(1,1)$ with vertex along the subspace $\{y_{12}=y_{21}=y_{10}=y_{20}=0\}$,   the linear image of $TP^2$.\end{example}

\subsection{Veronese varieties}
\label{sec:veronese-varieties}
Consider the \emph{$2$--Veronese variety} $V_{2,n}$ of quadrics in $\P^ n$
embedded in $\P^ {\frac {n(n+3)}2}$ with coordinates $[x_{ij}]_ {0\leqslant
i\leqslant j\leqslant n}$.  The following map $\phi$ is a linearizing affine $(2,2)$ Cremona
transformation  for $V_{2,n}$ defined on $\{x_{00}\not=0\}$
     \[
    (x_{ij})_{0\leq i\le j\leq n}\mapsto (y_{ij})_{0\leq i\le j\leq n}=
     (x_{01},\dots, x_{0n}, x_{ij}-x_{0i}x_{0j})_{1\leq i\le j\leq n}.
     \]
 Its reduced fundamental locus is $\{x_{00}=0\}$ and the indeterminacy locus is $\{x_{00}=\ldots=x_{0n}=0\}$.
     
We interpret $V_{2,n}$ as the set of rank $1$ symmetric matrices
$x=(x_{ij})_ {0\leqslant
i, j\leqslant n}$  with  $x_{ji}=x_{ij}$ if $j<i$.
The $(k-1)$--secant variety to $V_{2,n}$ is
defined by the $k\times k$--minors of $x$.  On $\{x_{00}\not=0\}$ the rank of $x$
coincides with the rank of
 \[  
      \begin{pmatrix} 1&0&0&\ldots &0\\
 x_{01}&x_{11}-x_{01}^2&x_{12}-x_{01}x_{02}&\ldots &x_{1n}-x_{01}x_{0n}\\
  \vdots& \vdots& \vdots& \vdots&\vdots\\
  x_{0n}&x_{1n}-x_{01}x_{0n}&x_{2n}-x_{02}x_{0n}&\ldots &x_{nn}-x_{0n}^2
  \end{pmatrix}=    
    \begin{pmatrix} 1&0&0&\ldots &0\\
 x_{01}&y_{11}&y_{12}&\ldots &y_{1n}\\
  \vdots&\vdots&\vdots&\vdots&\vdots\\
  x_{0n}&y_{1n}&y_{2n}&\ldots &y_{nn}
  \end{pmatrix},
  \] 
So the $(k-1)$--secant variety to  $V_{2,n}$ is mapped by $\phi$ 
    to the cone over the $(k-2)$--secant variety of $V_{2,n-1}$ with
    vertex along the linear image of $V_{2,n}$ in $\P^ {\frac {n(n+3)}2}$.

\begin{example}\label{ex:ver} The secant variety $\Sec(V_{2,2})$ of the Veronese surface in $\P^5$ is mapped to the cone over the conic $V_{2,1}=\{y_{11}y_{22}-y_{12}^2=0\} \subset\P^2=\{y_{00}=y_{10}=y_{20}=0\}$ with vertex $\P^2=\{y_{11}=y_{12}=y_{22}=0\}$.\end{example}

In general, the \emph{$d$--Veronese variety} $V_{d,n}$ of $\P^ n$ is
embedded in $\P^ {\binom{n+d}n-1}$ with coordinates $[x_{i_0\ldots
  i_n}]_ {i_0+\ldots+ i_n=d}$, with $i_j\geqslant 0$ for $ 0\leqslant
j\leqslant n$.  Its projection from the linear space $\{x_{i_0\ldots
  i_n}=0\}_{i_0\geqslant d-1}$ to the $n$--space $\{x_{i_0\ldots
  i_n}=0\}_{i_0<d-1}$, is birational. Accordingly, we can find a
Cremona linearizing map. We will treat the curve case in Section~\ref
{sec:rnc} but we will not dwell on the higher dimensional and higher
degree cases.

\subsection{Grassmannians
  of lines}
\label{sec:grassm-vari}
In this section we present Cremona linearizations of Grassmannians of
lines. Analogous linearizations exist for higher Grassmannians, an
example of which we treat in Section~\ref {ssec:grass}.

Let $V$ be a complex  vector space of dimension $n$. 
We can identify $V$ with $\C^n$, once we fix a basis $({\bf e}_0,\ldots, {\bf e}_{n-1})$ of $V$. 
The \emph{Pl\"ucker embedding} maps the Grassmannian $G(2,n)$ of $2$--dimensional vector subspaces (i.e. \emph{$2$--planes}) of $V$  into $\P^{\frac{n(n-1)}{2}-1}=\P(\wedge^ 2V)$, which we  identify with the projective space associated to the vector space of antisymmetric matrices of order $n$, thus 
the coordinates are $[x_{ij}]_{0\leq i<j\leq n-1}$. 

Two vectors  
\[{\bf \xi}_0=(\xi_{00},\ldots, \xi_{0,n-1}),\,\,\, {\bf
  \xi}_1=(\xi_{10},\ldots, \xi_{1,n-1})\] in $V$ that span a
$2$--plane $W$ form the rows of a $2\times n$--matrix $x$, whose
minors are independent on the chosen points, up to a nonzero common
factor.  The \emph{Pl\"ucker point} associated to $W$ is
$[x_{ij}]_{0\leq i<j\leq n-1}$, where $x_{ij}$ denotes the minor of
$X$ obtained by choosing the $i$-th\ and $j$-th columns.

The \emph{Pl\"ucker ideal} $I_{2,n}$ is the homogeneous ideal of
$G(2,n)$ in its Pl\"ucker embedding. This ideal is prime and it is
generated by quadrics. More precisely, $I_{2,n}$ is generated by the
$\binom{n}{4}$ \emph{three terms P\"ucker relations}
\begin{equation}
x_{ij}x_{kl}-x_{ik}x_{jl}+x_{il}x_{jk} \qquad \text{ for } 0\leqslant
i<j<k<l\leqslant n-1.\label{eq:PluckerRel}
\end{equation}

Using~\eqref{eq:PluckerRel}, in the open affine $\{x_{01}\neq 0\}$, we
have parametric equations for $G(2,n)$: the parameters are the $2n-4$
coordinates $x_{ij}$ with $i=0,1$, and the equations for the remaining
coordinates are
\[
x_{ij}=x_{0i}x_{1j}-x_{0j}x_{1i},\,\,\, \text {for}\,\,\, 2\leqslant i<j\leqslant n-1.
\]
Hence $G(2,n)$ is rational, and a birational map $G(2,n)\dasharrow \P^
{2n-2}$ is given by projecting $G(2,n)$ from the linear span $\P^
{\frac {n(n-5)}2+2}$ of $G(2,n-2)$ viewed inside $G(2,n)$ 
as the 
Grassmannian of 2--planes in $V'=\langle {\bf
  e}_2, \ldots, {\bf e}_{n-1}\rangle\subset V$.

According to Section \ref {ssec:lincrem}, we have a triangular $(2,2)$--Cremona linearization 
 $\varphi\colon
\P^{\frac{n(n-1)}{2}-1}\dashrightarrow \P^{\frac{n(n-1)}{2}-1}$ of $G(2,n)$, given in affine coordinates by
\begin{equation*}
y_{ij}=  
\begin{cases}
 x_{ij} &\text{ if } i=0,1,\; 2\leq j\leq n-1,\\
x_{ij}-x_{0i}x_{1j}+x_{0j}x_{1i} &\text{ if } \,\,\, 2\leqslant i<j\leqslant n-1.
\end{cases}
\end{equation*}
The reduced fundamental locus is $\{x_{01}=0\}$, and the indeterminacy
locus is the union of the two linear spaces 
$\{x_{01}=x_{02}=\ldots=x_{0(n-1)}=0\}$ and $\{x_{01}=x_{12}=\ldots =x_{1(n-1)}=0\}$.
    
On the complement of $\{x_{01}= 0\}$, the Grassmannian $G(2,n)$ is the
set of rank $2$ matrices of the form
     \[
  x=\begin{pmatrix} 0&1&x_{02}&x_{03}&\ldots&x_{0n}\\ -1
       &0&x_{12}&x_{13}&\ldots &x_{1n}\\ -x_{02}
       &-x_{12}&0&\ddots &\vdots&\vdots \\ -x_{03}
       &-x_{13}&\ddots&\ddots&\ddots&\vdots\\ \vdots
       &\vdots&\vdots&\ddots&\ddots &x_{n-2,n-1}
       \\-x_{0,n-1}&-x_{1,n-1}&-x_{2,n-1}&\vdots&-x_{n-2,n-1}&0\end{pmatrix}.
    \]
Performing suitable column operations on
 $x$ and using the $y$--coordinates, we see that the rank of $x$  is 2 plus the rank of the matrix
\[\begin{pmatrix} 
0&y_{23}& \ldots&\ldots&y_{2,n-1}\\ 
 \vdots &\vdots&\vdots&\vdots&\vdots\\\
   \vdots &\vdots&\vdots&0&y_{n-2,n-1}\\
-y_{2,n-1}&\ldots&\ldots&-y_{n-2,n-1}&0\end{pmatrix}.
    \]
   
Since $\Sec_k(G(2,n))$ is the set of antisymmetric matrices of rank $2k+2$, we see that $\Sec_k(G(2,n))$
is mapped by $\varphi$  to the cone over $\Sec_{k-1}(G(2,n-2))$ with vertex along the linear image of
$G(2,n)$.

\begin{example}\label{ex:grass} The first secant and tangent variety of $G(2,6)$ coincide and are defined by the Pfaffian cubic polynomial
\begin{align*}
&      x_{01}(x_{23}x_{45}- x_{24}x_{35}+ x_{25}x_{34})-x_{02}(x_{13}x_{45}
      -x_{14}x_{35}+x_{15}x_{34})+x_{03}(x_{12}x_{45}- x_{14}x_{25}+
      x_{15}x_{24})
    \\
  &    -x_{04}(x_{12}x_{35}-x_{13}x_{25}+x_{15}x_{23})+x_{05}(x_{14}x_{23}-x_{13}x_{24}+x_{12}x_{34})=0.\label{eq:2}
    \end{align*}
    In the $y$--coordinates, this hypersurface has a much
    simpler equation, namely the Pl\"ucker equation of $G(2,4)$
    \[ y_{23}y_{45}- y_{24}y_{35}+y_{25}y_{34}=0.
    \]
\vspace{-5ex}
\end{example}

\vspace{3ex}

\begin{example} A different Cremona transformation that linearizes $G(2,n)$ was considered in \cite{GHPRS}, namely
%
    \[
   y_{0i}\mapsto \frac{1}{x_{0i}}\quad i=1,\ldots,n-1, \qquad y_{ij}\mapsto \frac{x_{ij}}{x_{0i}x_{0j}}\quad 1\leqslant i<j\leqslant n-1.    \]
  It maps $G(2,n)$ to the linear space defined by 
   \[ y_{ij}-y_{ik}+y_{jk}=0 \qquad 1\leqslant i<j<k\leqslant n-1.
   \]
    This transformation is studied   in \cite{GHPRS} to compare various
    notions of convexity for lines.    
    \end{example}

  \subsection{Severi varieties}
  \label{sec:severi-varieties}
  The Veronese surface $V_{2,2}$, the Segre variety $\Segre(2,2)$ and
  the Grassmannian $G(2,6)$ mentioned in Examples \ref {ex:ver}, \ref
  {ex:seg} and \ref {ex:grass} are \emph{Severi varieties},
  i.e. smooth 1--defective varieties of dimension $n$ in $\P^ {\frac
    32n+2}$ (see \cite {Zak1}).  There is one more Severi variety: the
  so--called \emph{Cartan variety} of dimension 16 embedded in $\P^
  {26}$.
      
  Let $X$ be a Severi variety. It is known that $X$ is swept out by a
  $n$--dimensional family $\mathcal Q$ of $\frac n2$--dimensional
  smooth quadrics, such that, given two distinct points $x,y\in X$,
  there is a unique quadric of $\mathcal Q$ containing $x,y$.  If
  $Q\in \mathcal Q$, the projection of $X$ from the linear space
  $\langle Q\rangle$ of dimension $\frac n2+1$ to $\P^ n$ is
  birational and, as usual by now, we get a Cremona linearization $\phi$ of $X$. 
  
  Being $X$ defective, its tangent and   first secant varieties coincide. 
  By Example~\ref{ex:tan} we see
      that $\phi$ maps $T(X)={\rm Sec}(X)$ to the cone over $Q$ with vertex the $n$--dimensional
      linear image of $X$. This agrees with the
      contents of the previous examples and applies to the Cartan
      variety as well.

   \subsection{Rational normal curves}
\label{sec:rnc}

Let $V_n:=V_{1,n}$ be the rational normal curve of degree $n$ in $\P ^n$
\[
V_n=\{[t^n, st^{n-1}, \ldots, s^{n-1}t, s^n] \,:\, [s,t]\in \P^1\} \subset \P^n.
\]
Let $[x_0, \ldots, x_n]$ be the coordinates of $\P^n$. Then
$V_n$ is the determinantal variety
           \[
\rk \begin{pmatrix} x_0&x_1&\ldots &x_{n-1}\\  
  x_1&x_2&\ldots&x_n\\  
 \end{pmatrix}=1.
    \]

Assume $n=2k$ is  even (similar considerations can be made in the odd case). Then, we can linearize $V_n$ via the
following affine triangular $(2,2)$--Cremona map $\phi$ on $\{x_0\neq 0\}$
\[
y_{i}=
\begin{cases}
x_1 & \text{ if } i=1,\\
  x_i-x_{i-1}x_1 &\text{ if }i>1 \text { and } i \text{ is odd},\\
x_i-x_{\frac i2}^2  &\text{ otherwise.}
\end{cases}
\]
The $\phi$--image of $V_n$ is the linear space $\{y_{2}=y_3=\ldots=y_n=0\}$.
The reduced fundamental locus is $\{x_0=0\}$ and the indeterminacy locus is $\{x_0=x_1=\ldots=x_{k}=0\}$

The secant variety $\Sec(V_n)$ is defined by the $3\times 3$-minors of  the $3\times(n-1)$ \emph{catalecticant} matrix
 \[
\begin{pmatrix} x_0&x_1&\ldots &x_{n-2}\\  
 x_1&x_2&\ldots&x_{n-1}\\
  x_2&x_3&\ldots&x_{n}\\  
 \end{pmatrix},
    \]
   (see \cite {dol}), where we as usually set $x_0=1$. Using column operations, this matrix can be transformed into the following one expressed in terms of the $y$--coordinates
    \[
\begin{pmatrix} 1&0&\ldots &0\\  
 y_1&y_2&\ldots&y_{n-1}\\
  y_2+y_1^ 2&y_3&\ldots&y_{n}\\  
 \end{pmatrix}.
    \]
    This shows that $\Sec(V_n)$ is mapped by $\phi$ to the cone over a
    $V_{n-2}$ with vertex the line $\{y_2=\ldots=y_n=0\}$ which is the
    $\phi$--image of $V_n$. A similar situation occurs for all higher
    secant varieties of $V_n$. For instance, $\Sec_{k-1}(V_n)$ is the
    hypersurface defined by the catalecticant determinantal equation
    of degree $k+1=\frac n2+1$
  
   \[
\det\begin{pmatrix} x_0&x_1&\ldots &x_{k}\\  
 x_1&x_2&\ldots&x_{k+1}\\
 \ldots& \ldots& \ldots& \ldots\\
 x_k&x_{k+1}&\ldots&x_{n}\\  
 \end{pmatrix}=0.
    \] 
    Hence, $\phi$ maps $\Sec_{k-1}(V_n)$ to the cone over
    $\Sec_{k-2}(V_{n-2})$ with vertex the $\P^1$ obtained as the
    $\phi$--image of $V_n$.

  \subsection{The Grassmannians $G(3,6)$}\label {ssec:grass}
Let  $X=(x_{ij})_{1\leqslant i,j\leqslant 3}$ and $Y=(y_{ij})_{1\leqslant i,j\leqslant 3}$ be $3\times 3$-matrices, and let  
\[
[ x_0,X,Y,y_0]
\] 
be coordinates in $\P^{19}$.

Let $A=(a_{ij})_{1\leqslant i,j\leqslant 3}$ be a $3\times 3$ matrix.  We denote by $A_{ij}$ the minor of $A$ obtained by deleting row $i$ and column $j$, so that
\[ \wedge ^ 2A= (A_{ij})_{1\leqslant i,j\leqslant 3},\,\,\, \text {and}\,\,\, \wedge^ 3A=\det (A).\]
We  parametrize  $G(3,6)$ as follows:
    \[
    (I_3|A)\in \C^ 9\mapsto (1, A, \wedge^2 A, \wedge^3
A)\in \{x_{0}\not=0\}\subset\P^{19}.
\]
This parametrization is precisely the inverse of the birational projection of $G(3,6)$ from its tangent space at the point $[0,0,0,1]$.

By our discussion in Section~\ref {ssec:lincrem}, this gives rise to a
family of triangular Cremona transformations linearizing $G(3,6)$
\[\phi: [ x_0,X,Y,y_0]\in \P^{19}\dasharrow [z_0, Z,W,w_0]\in \P^{19}\]
where $Z=(z_{ij})_{1\leqslant i,j\leqslant 3}$ and $W=(w_{ij})_{1\leqslant i,j\leqslant 3}$.

We can view the determinant of $A$ in two ways: as a cubic polynomial
in the entries of $A$ and as a bilinear quadric form in the variables
$(a_{ij}, A_{ij})$. This yields different Cremona transformations, one
defined by quadrics whose inverse transformation is defined by cubics
(a \emph{quadro--cubic} transformation), the other defined by cubics
with the inverse also defined by cubics (a \emph{{cubo--}cubic}
transformation).

Let us start with the quadro-cubic transformation $\phi$.  On
$\{x_0\not=0\}$ it is defined by
 \begin{equation*}
  \label{eq:quadroG36}
   z_{ij}=x_{ij},\quad w_{ij}=y_{ij}-X_{ij}, \quad
  w_{0}=y_{0} - \sum_{i=1}^3 (-1)^{i+1} x_{1i}y_{1i}.
\end{equation*}
The reduced fundamental locus is $\{x_0=0\}$ and the indeterminacy locus is $\{x_0=x_{ij}=0\}$.
The inverse of $\phi$, on $\{z_0\not=0\}$,  is given by
\begin{equation}\label{eq:inv}
x_{ij}=z_{ij}, \quad y_{ij}=w_{ij}+Z_{ij}, \quad  y_0=w_0+\sum_{i=1}^3 (-1)^{i+1} z_{1i}(w_{1i}+Z_{1i}).
\end{equation}

The  {cubo--}cubic Cremona transformation $\psi$  is given on the affine set $\{x_0\not=0\}$ by the following expressions
\begin{equation*}
  \label{eq:cuboG36}
z_{ij}=x_{ij},\quad w_{ij}=y_{ij}-X_{ij}, \quad
  w_{0}=y_{0} - \det(X).
\end{equation*}
Its reduced fundamental locus is $\{x_0=0\}$, and its  indeterminacy locus is $\{x_0=x_{ij}=0\}$.
The inverse Cremona transformation restricted to $\{z_0\not=0\}$ is defined by
\begin{equation*}
  \label{eq:invcuboG36}
 x_{ij}=z_{ij},\quad y_{ij}=w_{ij}+Z_{ij}, \quad
  y_{0}=w_{0} + \det(Z).
\end{equation*}
The image of $G(3,6)$ under both $\phi$ and $\psi$ is the linear space
defined by $\{W=0, w_0=0\}$.

It is known that  $\Sec(G(3,6))=\P^{19}$ (see \cite {donagi77}),  while $T(G(3,6))$ is the quartic hypersurface 
defined by 
 \[
 P=(x_{0}y_{0}-{\rm tr} (XY))^2+4x_{0}\det(Y)+4y_{0}\det(X)-4\sum_{i,j}\det(X_{ij})\det(Y_{ji})=0,
 \]
 (see \cite[p. 83]{SK}).  We find the equation of $\phi(T(G(3,6)))$ by
 plugging \eqref {eq:inv} in $P$ (where $x_0=1$). We obtain a degree
 6 equation
\[
z_{13}^4z_{22}^2-2z_{12}z_{13}^3z_{22}z_{23}+ \quad {\rm appr.}\;\,\, 600 \quad {\rm terms}=0.
\]

Analogously, for  $\psi(T(G(3,6)))$ we obtain
\[
z_{13}^4z_{22}^2-2z_{12}z_{13}^3z_{22}z_{23}+ \quad {\rm appr.}\;\,\,   600 \quad {\rm terms}=0.
\]

Since $T(G(3,6))$ is singular along $G(3,6)$, the same happens for the
above two sextics along $\{W=0, w_0=0\}$. In any event, none of these
two linearizing Cremonas simplify the equation of $T(G(3,6))$, which
actually becomes more complicated.

Similar considerations can be done for Grassmannians $G(n,2n)$ with
$n\geq 4$.

 \section{Cumulant Cremonas}\label{sec:cum}
 
 As we saw in Section \ref{sec:examples}, there are several examples
 in which a Cremona linearization of a rational variety simplifies the
 equations of its secant varieties.  Here is another instance of this
 behavior.
 
 \begin{example}\label{ex:3segre} 
   Consider the Segre embedding $\Sigma_n$ of $(\P^1)^ n$ in $\P^{2^
     n-1}$.  In particular take the case $n=3$. Then, $\Sigma_3$ is
   parametrically given by the equations
   \[x_1=t_1,\quad x_2=t_2,\quad x_3=t_3,\quad x_4=t_1t_2, \quad
   x_5=t_1t_3, \quad x_6=t_2t_3, \quad x_7=t_1t_2t_3.\] We have
   $\Sec(\Sigma_3)=\P^ 7$, whereas $T(\Sigma_3)$ is a hypersurface of
   degree four in $\P^7$. Its defining equation is the so called
   \emph{hyperdeterminant} (see \cite {GKZ}).

   The linearizing Cremona transformation $\phi$ defined in
   (\ref{eq:basiccremona}) maps $\Sigma_3$ to the linear space
   $x_4=\ldots=x_7=0$. Following (\ref{eq:param4}), the variety
   $T(\Sigma_n)$ is mapped under $\phi$ to a (symmetric) degree four
   hypersurface with defining equation
 \[x_3^ 2x_4^ 2+ x_2^ 2x_5^ 2+x_1^ 2x_6^ 2+ 2(x_1x_2x_5x_6+
 x_1x_3x_4x_6+ x_2x_3x_4x_5)+ 4x_4x_5x_6
 -2x_7(x_1x_6+x_3x_4+x_2x_5)+x_7^ 2=0.\]   
\vspace{-5ex}\end{example}

In this case the
 linearization process simplifies the equation of $T(\Sigma_3)$, but the degree remains the same.  The question is:
 can we find a linearizing Cremona for $\Sigma_3$ which lowers the degree of $T(\Sigma_3)$? An affirmative answer to this question is given by \emph{cumulant Cremonas} arising from algebraic
statistics. Indeed,
 this family of Cremonas gives the following very simple equation for the image of $T(\Sigma_3)$ (see \cite[(2.1)]{pwz-2011-bincum})
\[x_7^ 2+4x_4x_5x_6=0.\]

\subsection{Binary cumulants}

 We  recall the setting of binary cumulants from~\cite{pwz-2011-bincum}.  Let
 $\Pi(I)$ denote the set of all nonempty set partitions of
 $I\subseteq[n]:=\{1,\ldots,n\}$. We write $\pi=B_{1}|\cdots|B_{r}$
 for a typical element of $\Pi(I)$, where all $\emptyset\neq B_{i}\subset I$ are
 the unordered disjoint \emph{blocks} of $\pi$, with $I=\cup_{i=1}^ rB_i$. For example, if $n=3$, then
$$
\Pi([3])=\{123,\;1|23,\; 2|13,\; 3|12,\; 1|2|3\}.
$$ 
We denote by $|\pi|$ the number of blocks of $\pi \in \Pi(I)$.

Consider two copies of
$\P^{2^{n}-1}=\P(\C^{2}\otimes\cdots\otimes\C^{2})$ with coordinates  $[x_{I}]_{I\subseteq [n]}$
and $[y_{I}]_{I\subseteq [n]}$.  Following 
\cite{pwz-2011-bincum}, we define the \emph{(binary) cumulant Cremona transformation}  or the  \emph{(binary) cumulant change of coordinates} 
  \[\psi\colon [x_{I}]_{I\subseteq [n]}\in \mathbb{P}^{2^{n}-1}\dasharrow [y_{I}]_{I\subseteq [n]}\in 
\mathbb{P}^{2^{n}-1}\] 
via the formula
\begin{equation}\label{eq:cumul}
 y_{\emptyset}=x_{\emptyset}^{n}, 
\quad \,\,\, \text{and}\,\,\,\quad
y_{I}=\sum_{\pi\in \Pi(I)}(-1)^{|\pi|-1}(|\pi|-1)!\;x_{\emptyset}^{n-|\pi|}\prod_{B\in\pi} x_{B} \qquad \text{for }\emptyset \neq I\subseteq [n],
\end{equation}
where the product
in~\eqref{eq:cumul} is taken over all blocks $B$ of $\pi$ (we will call $[y_{I}]_{I\subseteq [n]}$ the \emph{cumulant coordinates}).  
 Note that $I$ is the
maximal element in the poset $\Pi(I)$, hence $\psi$ is a triangular
Cremona transformation. It linearizes $\Sigma_n$, which is mapped to the linear space
$\{y_I=0\}_{|I|\geqslant 2}$ (see \cite[Remark
3.4]{pwz-2011-bincum}), and $T(\Sigma_n)$ is 
 toric in the cumulants coordinates (see \cite[Theorem 4.1]{pwz-2011-bincum}).

The inverse map of $\psi$ is given by the standard M\"{o}bius
inversion formula for the partition lattice $\Pi([n])$  
(see \cite[Proposition 3.7.1]{stanley2006enumerative})
\begin{equation*}\label{eq:inversecumul}
  x_{\emptyset}=y_{\emptyset}^{n},\quad \,\,\, \text{and}\,\,\,\quad
  x_{I}=\sum_{\pi\in \Pi(I)}y_{\emptyset}^{n-|\pi|}\prod_{B\in\pi} y_{B}\qquad\text{for }I\subseteq [n].
\end{equation*}
Both maps are morphisms on the open affine subsets
$\{x_{\emptyset}\neq 0\}$ and $\{y_{\emptyset}\neq 0\}$, respectively.

\begin{example}
  Fix $n=2$. Then 
  \[y_{\emptyset}=x_{\emptyset}^{2}, \,\,\, 
  y_{1}=x_{\emptyset}x_{1}, \,\,\, y_{2}=x_{\emptyset}x_{2}\\, \,\,\, y_{12}=x_{\emptyset}x_{12}-x_{1}x_{2},\]
  which coincides with \eqref{eq:SegreNN} in this case. 
The inverse is given by
  \[x_{\emptyset}=y_{\emptyset}^{2}, \,\,\, x_{1}=y_{\emptyset}y_{1},
  x_{2}=y_{\emptyset}y_{2},\,\,\, x_{12}=y_{\emptyset}y_{12}+y_{1}y_{2}.\]
The {fundamental locus} is $\{x_{\emptyset}^{3}=0\}$.

  Let  $n=3$. Then
  \[
\begin{array}{lr}
y_{\emptyset}=x_{\emptyset}^{3},\,\,\, 
  y_{i}=x_{\emptyset}^{2}x_{i}, \,\,\, \text {for}\,\,\, 1\leqslant i\leqslant 3, \,\,\,  
  y_{ij}=x_{\emptyset}^{2}x_{ij}-x_{\emptyset}x_{i}x_{j}, \,\,\, \text {for}\,\,\,   1\leq
i<j\leq 3\\
y_{123}=x_{\emptyset}^{2}x_{123}-x_{\emptyset}x_{1}x_{23}-x_{\emptyset}x_{2}x_{13}-x_{\emptyset}x_{3}x_{12}+2x_{1}x_{2}x_{3}.\\
\end{array}
\]
The inverse is
\[
\begin{array}{lr}
x_{\emptyset}=y_{\emptyset}^3,\,\,\, 
  x_{i}=y_{\emptyset}^2y_i, \,\,\, \text {for}\,\,\, 1\leqslant i\leqslant 3, \,\,\,  
  x_{ij} = y_{\emptyset}(y_{\emptyset}y_{ij}+y_iy_j), \,\,\, \text {for}\,\,\,   1\leq
i<j\leq 3\\
x_{123}=y_{\emptyset}^2y_{123}+ y_{\emptyset}
(y_{12}y_3+y_{13}y_2+y_{23}y_1)+ y_1y_2y_3.\\
\end{array}
\]
The {fundamental locus}  is now
$\{x_{\emptyset}^{8}=0\}$.
\end{example} 

\subsection{Linearizing higher Segre varieties}
The above construction can be generalized to
$\Segre(r_1,\ldots,r_k)\subset \P^ r$ with $r+1=\prod_{i=1}^
3(r_i+1)$, for any $k\geq 2$ and $r_1\geqslant \ldots \geqslant
r_k\geqslant 1$.  The case $k=2$ has been treated in Section
\ref{sec:segre-varieties}. If $k=3$, let  $[x_{ijk}]_{0\leq i\leq
  r_1,0\leq j\leq r_2,0\leq k\leq r_3}$ be the coordinates on $\P^r$. Define a Cremona transformation by
\begin{eqnarray*}
& y_{000}=x_{000}^3,\quad y_{i00}=x_{000}^2x_{i00},\quad y_{0j0}=x_{000}^2x_{0j0},\quad y_{00k}=x_{000}^2x_{00k},&\\
&y_{ij0}=x_{000}(x_{000}x_{ij0}-x_{i00}x_{0j0}),\quad y_{i0k}=x_{000}(x_{000}x_{i0k}-x_{i00}x_{00k}),\quad y_{0jk}=x_{000}(x_{000}x_{0jk}-x_{0j0}x_{00k}),\quad &\\
& y_{ijk}=x_{000}^2x_{ijk}-x_{000}x_{i00}x_{0jk}-x_{000}x_{0j0}x_{i0k}-x_{000}x_{00k}x_{ij0}+2x_{i00}x_{0j0}x_{00k}, &
\end{eqnarray*}
where $i,j,k\geq 1$. This  linearizes ${\rm Seg}(r_1,r_2,r_3)$  by mapping it to the linear space  $\{y_{ijk}=0\}$ for all triples $(i,j,k)\in \prod _{i=1}^ 3\{0,\ldots,r_i\}$ with at least two nonzero coordinates. 

This generalizes to any $k$ as follows (see \cite[Sections 7 and
8]{MOZ}).  Let $S(\mathbf{i})\subseteq [n]$ be the support of
$\mathbf{i}=(i_1,\ldots,i_k)\in \prod_{i=1}^k \{0,\ldots,r_i\}$, i.e.\
the set of coordinates of nonzero entries in $\mathbf{i}$. For every
$B\subseteq[k]$, we define the $k$--tuple $\mathbf{i}(B)$ that agrees
with $\mathbf{i}$ on those indices in $B$ and is zero otherwise.  We
define the Cremona transformation $\psi\colon \P^{r}\dashrightarrow
\P^{r}$ by the formulas
$$
y_\mathbf{i}=\sum_{\pi\in\Pi(S(\mathbf{i}))}(-1)^{|\pi|-1}(|\pi|-1)!\,\, x_{0\cdots 0}^{n-|\pi|}\prod_{B\in \pi} x_{\mathbf{i}(B)},\,\,\, \text{for all} \,\,\, \mathbf{i}\in \prod_{i=1}^k
\{0,\ldots,k_i\}.
$$  
The image of ${\rm Seg}(r_1,\ldots,r_k)$
lies in the subspace $\{y_\mathbf{i}=0\}_{|S({\bf i})|\geqslant 2}$. This can be shown by mimicking the proof of Theorem \ref{thm:lincrem} below, so we leave the proof to the reader.

\subsection{$\mathcal L$-cumulant Cremonas}

One of the advantages of working with cumulants is that the change of
coordinates is conveniently encoded by the cumulant generating
function \cite{pwz-2011-bincum}. However, to fully exploit the involved
combinatorics, we will generalize cumulants to situations
when such a generating function is not known. As we will see,
$\mathcal L$-cumulants, introduced in \cite {pwz-2010-cumulants}, enjoy this property.

First we show how the homogeneous binary cumulant change of
coordinates generalizes. We replace the partition lattice $\Pi(I)$ by
a \emph{partial order set (poset)} $(P, <_P)$ (or simply $(P,<)$ if
there is no danger of confusion) with its associated \emph{M\"obius
  function} $\mu_P$. The function $\mu_P\colon P\times P\rightarrow
\mathbb{Z}$ (or simply $\mu$) is recursively defined by
$\mu(\pi,\pi)=1$ for all $\pi \in P$, $\mu(\pi, \nu)=0$ if $\pi \not<
\nu$, and
$$
\mu(\pi,\nu)=-\sum_{\pi\leq\delta<\nu}\mu(\pi,\delta),\quad\text{for
  all }\pi< \nu \text{ in } P.
$$
The two main features of this function that we will use in the rest of
this section are the \emph{M\"obius inversion formula} and the
\emph{product theorem}, which we now recall. Even though they hold in
a more general setting, we state them for finite posets, since this
will suffice for our purposes.
\begin{proposition}[\textbf{M\"obius inversion  formula
  }~{\cite[Proposition
  3.7.1]{stanley2006enumerative}}] \label{pr:MobiusInversion}
  Let $(P, <)$ be a finite poset and $f,g\colon P\to \C$. Then
\[
g(x)=\sum_{y\leq x} f(y)\quad \text{ for all } x\in P \qquad \text{
  if and only if } \qquad f(x)=\sum_{y\leq x} g(y)\mu(y,x) \quad
\text{ for all } x\in P.
\]
\end{proposition}

\begin{theorem}[\textbf{Product theorem}~{\cite[Proposition
     3.8.2]{stanley2006enumerative}}]
    \label{thm:ProductTheorem}
    Let $(P, <_P)$ and $(Q, <_Q)$ be finite posets and let 
    $(P\times Q, <)$ be their product, with order given coordinatewise,
    i.e.\ $(p,q)\leq (p', q')$ if and only if $p\leq_P p'$ and
    $q\leq_Q q'$. If $(p,q)\leq (p', q')$ in $P\times Q$, then
\[
\mu_{P\times Q}((p,q), (p',q'))=\mu_{P}(p,p')\,\mu_Q(q,q').
\]
\end{theorem}
\noindent 
For further basic results concerning M\"{o}bius functions we refer the reader
to \cite[Chapter 3]{stanley2006enumerative}. Some of them will be
recalled later on in this section.

The set partitions of a given nonempty set form a poset, where the order $<$
corresponds to refinement, that is, $\pi<
\nu$ if $\pi$ refines $\nu$. To generalize cumulant Cremonas, we replace
$\Pi([n])$ by a subposet $\mathcal{L}$ containing the maximal
and minimal elements of $\Pi([n])$, i.e., the partitions 
\[
\hat{0}=1\vert \ldots\vert n\,\,\, \text {and}\,\,\,
\hat{1}=[n].
\]
These two elements coincide if and only if $n=1$. 

Let us fix such an $\mathcal{L}$.  For each $I\subseteq [n]$, we
construct a subposet $\mathcal{L}(I)$ of $\Pi(I)$ by restricting each
partition in $\mathcal{L}$ to $I$. In particular,
$\mathcal{L}([n])=\mathcal{L}$.  Each poset $\mathcal{L}(I)$ has an
associated M\"obius function. To simplify notation, we denote all of
them by $\mu$. Similarly, we denote the maximal and the minimal
element of each poset by $\hat{0}$ and $\hat{1}$, so $\hat{1}=I$ in
$\mathcal{L}(I)$. The identification will be clear from the context.

Given $\mathcal{L}$, we define a map $\psi_{\mathcal{L}}\colon
\P^{2^{n}-1}\dasharrow \P^{2^{n}-1}$ as
\begin{equation}\label{eq:cremonagen}
  y_{I}=  \begin{cases}
\sum_{\pi\in
      \mathcal{L}(I)}\mu(\pi,\hat{1})\;x_{\emptyset}^{n-|\pi|}\prod_{B\in\pi}
    x_{B}& \quad \text{if }I\neq \emptyset,\\
x_{\emptyset}^{n} & \quad \text{otherwise.}\\
  \end{cases}
\end{equation}
Here, $B\in \pi$ if it is a block of the partition.  Note that 
\[
y_i=x_\emptyset ^ {n-1} x_i, \,\,\, \text {for all}\,\,\, i\in [n],\,\,\, \text {and},\,\,\, y_{ij} = x_\emptyset ^ {n-2}(x_\emptyset x_{ij}-x_ix_j),
\]
do not depend on $\mathcal{L}$. Since $\hat{1}\in \mathcal{L}$, we
know that $I$ is the maximal element of $\mathcal{L}(I)$ for every
$I\subset [n]$. This implies that $\psi_{\mathcal{L}}$ is a triangular
Cremona transformation. It is defined over the open set
$\{x_{\emptyset}\neq 0\}$. We call such a map an \emph{$\mathcal
  L$-cumulant Cremona}.  Its fundamental locus is
$\{x_\emptyset^{n^2-1}=0\}$.

\begin{example} If
$\mathcal{L}=\Pi([n])$, the M\"obius function satisfies
$\mu(\pi,\hat{1})=(-1)^{|\pi|-1}(|\pi|-1)!$, so we recover the
cumulant change of coordinates in~\eqref{eq:cumul}. 

To the other extreme, if $n>1$ and $\mathcal{L}=\{\hat{0},\hat{1}\}$, then \eqref {eq:cremonagen} becomes
\begin{equation*}\label{eq:cremonagenspec}
  y_{I}=  \begin{cases}
 x_\emptyset^ {n-1} x_I-x_\emptyset^ {n-|I|} \prod_{i\in I}  x_i
& \quad \text{if }I\neq \emptyset,\\
x_{\emptyset}^{n} & \quad \text{otherwise},\\
  \end{cases}
\end{equation*}
which is  the linearizing Cremona of $\Sigma_n$  arising, as in \eqref {eq:basiccremona}, from the affine parametrization of $\Sigma_n$ given by
\[
x_I= \prod_{i\in I}  t_i, \,\,\, \text {for}\,\,\, \emptyset \neq I\subseteq [n]\,\,\, \text {and} \,\,\, (t_1,\ldots,t_n)\in \C^ n.
\]\vspace{-5ex}

\end{example}

\begin{example}[\textbf{Interval partitions of
    {$[n]$}}]\label{ex:interval} Fix a positive integer $n$
  and let $\mathcal{L}$ be the set of \emph{interval partitions} of
  $[n]$, ordered by refinement.  An interval partition of $[n]$ is
  obtained by cutting the sequence $1,2,\ldots,n$ into
  subsequences. For example, there are four interval partitions on
  $[3]$, i.e., $123$, $1|23$, $12|3$ and
  $1|2|3$. 

  The interval partitions form a poset
  isomorphic to the Boolean lattice of a set of $n-1$ elements. In
  particular, its M\"{o}bius function satisfies
  $\mu(\pi,\hat{1})=(-1)^{|\pi|-1}$ (c.f.~\cite[Example
  3.8.3]{stanley2006enumerative}). If $n=3$, this gives the following
  formulas for the map $\psi_{\mathcal{L}}$ from~\eqref{eq:cremonagen}
\[
\begin{array}{c}
y_{\emptyset}=x_{\emptyset}^3, \quad
y_i = x_{\emptyset}^2x_i \; (i=1,2,3), \quad
y_{12}=x_{\emptyset}^{2}x_{12}-x_{\emptyset}x_{1}x_{2}, \quad y_{13}=x_{\emptyset}^{2}x_{13}-x_{\emptyset}x_{1}x_{3}, \quad y_{23}=x_{\emptyset}^{2}x_{23}-x_{\emptyset}x_{2}x_{3},\\ 
y_{123}=x_{\emptyset}^{2}x_{123}-x_{\emptyset}x_{1}x_{23}-x_{\emptyset}x_{3}x_{12}+x_{1}x_{2}x_{3}.
\end{array}
\]
\vspace{-5ex}

\end{example}

The formula for the inverse of $\psi_{\mathcal{L}}$ over
$\{y_{\emptyset}\neq 0\}$ follows by the standard M\"{o}bius inversion
formula on each poset $\mathcal{L}(I)$. Let us show how to do this. 
For every $\pi\in \Pi(I)$ we set
\begin{equation}\label{eq:x}
x_{\pi}:=\prod_{B\in \pi}x_B.\end{equation} 
Given  
$I\subseteq [n]$ and $\nu\in \mathcal{L}(I)$, we define
\begin{equation}\label{eq:ynu}
y_\nu:=\sum_{ \substack{\pi\leq \nu\\\ \pi\in \mathcal{L}(I)}}\mu(\pi,\nu)\, x_\pi.
\end{equation}
By
the M\"{o}bius inversion formula on $\mathcal{L}(I)$, we conclude that 
\begin{equation}\label{eq:partialInverse}
x_\nu=\sum_{\substack{\pi\leq \nu\\\pi\in \mathcal{L}(I)}}y_\pi\qquad
\text{ for all } I\subset [n]\,\,\, \text {and}\,\,\, \nu\in \mathcal{L}(I).
\end{equation}
In particular
\begin{equation}\label{eq:invLcum}
x_I=\sum_{\pi\in \mathcal{L}(I)}y_\pi,  \,\,\,\text{ for all } I\subset [n].
\end{equation}

The following lemma ensures that, for each $\nu \in \mathcal{L}(I)$,
$y_{\nu}$ is a polynomial in the variables $y_J$'s with $J\subseteq
I$. It also show that~\eqref{eq:invLcum} and yields an explicit
formula for $\psi_{\mathcal{L}}^{-1}$.

  \begin{lemma}\label{lm:recursiony_pi}
    For each $I\subset [n]$ and each $\nu\in \mathcal{L}(I)$, the
    variable $y_{\nu}$ is a polynomial in $y_{J}$'s where $J$ runs
    over all subsets of each one of the blocks of $\nu$.
  \end{lemma}
  \begin{proof}
We prove the result by induction on the subsets of $[n]$. If
$I=\emptyset$, there is nothing to prove since
$y_{\emptyset}=1$. Suppose that  $I\supsetneq \emptyset$ and that the
result holds for all $J\subsetneq I$.
    If $\nu$ is a one block partition, there is nothing to
    prove. Assume that $\nu$ contains more than one
    block. By ~\eqref {eq:x},~\eqref{eq:ynu} and~\eqref{eq:partialInverse} we
    obtain
\[
y_{\nu}=\sum_{\pi\leq \nu}\mu(\pi, \nu)\prod_{B\in \pi}(\sum_{\tau\in \mathcal{L}(B)}y_{\tau}).
\]
Since all $B$'s on the right-hand side are strictly
included in $I$, the result follows by induction.
  \end{proof}

As it happens with the homogeneous cumulant change of
coordinates~\eqref{eq:cumul}, the map
$\psi_{\mathcal{L}}$ linearizes  $\Sigma_n$:

\begin{theorem}\label{thm:lincrem}
  For any choice of $\mathcal{L}$, the map $\psi_{\mathcal{L}}$
  from~\eqref{eq:cremonagen} linearizes $\Sigma_n$. Its image is the
  linear space $\Pi:=\{y_{I}=0\}_{I\subseteq [n], |I|\geq 2}$.
\end{theorem}

\begin{proof}
  Denote by
  $a_{{i}}=[a_{i0},a_{i1}]$ the coordinates of the $i$--th copy of
  $\P^{1}$ in $(\P^{1})^{n}$.  The Segre embedding $\sigma_n\colon
  (\P^1)^n\to \P^{2^n-1}$, 
  maps $a=(a_{1}, \ldots, a_{n})$ to the point in $\P^{2^n-1}$
  whose $I$--th coordinate is  
  \[a_I=\prod_{i\in
    I}a_{i1}\prod_{i\notin I}a_{i0},\,\,\,\text{for every}\,\,\, I\subseteq [n].\]

  We compute $\psi_{\mathcal{L}} \circ
  \sigma_n$ using~\eqref{eq:cremonagen}. For every  $I \subseteq [n]$ and  every partition $\pi \in
  \mathcal{L}(I)$ we have 
$$ a_{\emptyset}^{n-|\pi|}\prod_{B\in \pi}a_{B}=\prod_{i\in I} (a_{i0}^{n-1}a_{i1}) \prod_{i\notin I} a_{i0}^{n}$$ 
which does not depend on $\pi$. Therefore, the $I$--th coordinate of
$\psi_{\mathcal{L}}(\sigma_n(a))$ is 
\begin{equation}\label{eq:lin}
 b_{I}=\big(\prod_{i\notin I} a_{i0}^{n}\prod_{i\in I} (a_{i0}^{n-1}a_{i1})\big) \sum_{\pi\in \mathcal{L}(I)} \mu(\pi,\hat{1}).
  \end{equation}
  If $|\mathcal{L}(I)|\geq 2$, Lemma~\ref{lm:PosetIdentity} below,
  applied to $P=\mathcal{L}(I)$, yields $\sum_{\pi\in \mathcal{L}(I)}
  \mu(\pi,\hat{1})=0$.  Combining this fact with~\eqref{eq:lin}, we
  conclude that the image of $\Sigma_n$ via $\psi_{\mathcal{L}}$ is
  contained in the linear space $\{y_I=0\}_{|\mathcal{L}(I)|\geq
    2}$. Note that since $\hat{1}$ and $\hat{0}$ lie in $\mathcal{L}$,
  the condition $| \mathcal{L}(I)|\geq 2$ is equivalent to $|I|\geq
  2$.  So this linear space is $\Pi$, and it has dimension
  $n$. Moreover $\Sigma_n$ is not contained in the fundamental locus
  of $\psi_\L$, so the induced map $\psi_{\L\vert \Sigma_n}\colon
  \Sigma_n\dasharrow \Pi$ is birational.
\end{proof}

\begin{lemma}\label{lm:PosetIdentity}  Let $(P, \leq)$ be a finite
  poset of size at least two with unique maximal and minimal elements
  $\hat{1}, \hat{0}$. Let $\mu$ be its M\"obius function. Then,
\[
\sum_{x\in P}
\mu(x,\hat{1})=0.
  \]
\end{lemma}
\begin{proof}
  Consider the dual poset $(P^*, {\leq}^*)$ obtained by reversing the order
  in $(P, \leq)$. In particular, the roles of the minimal and maximal
  elements are exchanged, namely $\hat{0}^*=\hat{1}$ and
  $\hat{1}^*=\hat{0}$. The M\"obius function $\mu^*$ of $P^ *$ satisfies
  $\mu^*(x,y)=\mu(y,x)$ for all $(x,y)\in P\times P$ (see \cite[page
  120]{stanley2006enumerative}). Therefore
\[
\sum_{x\in P} \mu(x,\hat{1})=\mu(\hat{0}, \hat{1}) +\sum_{\hat{0}<x\leq \hat{1}}\mu(x, \hat{1})=
\mu^*(\hat{0}^*, \hat{1}^*)+ \sum_{\hat{0}^*{\leq^*}\, x{<^*}\,\hat{1}^*}
  \mu^*(\hat{0}^*, x) =0,
\]
where the last equality follows from the recursive definition of
$\mu^*$.
\end{proof}

\subsection{Secant cumulants}

As we mentioned earlier, one of the useful features of binary
cumulants is that the tangential variety of $\Sigma_n$, expressed in cumulants, 
becomes toric. This is not the case, in general, for $\mathcal{L}$--cumulant
Cremonas.
However, with a careful choice of the defining poset $\mathcal{L}$
one may obtain other desired properties. For example, if
${\mathcal{L}}$ is the poset of interval partitions of $[n]$ defined
in Example~\ref{ex:interval}, the Cremona transformation
$\psi_{\mathcal{L}}$ is an involution.

 The next example is
related to  ${\rm Sec}(\Sigma_n)$  (see \cite {MOZ}, \cite[Section 3.3]{pwz-2010-cumulants}).
\begin{example}\label{ex:secant}
  In what follows, we parametrize the secant variety
  ${\Sec}(\Sigma_n)$ inside $\P^{2^n-1}$ starting from the
  parametrization of $\Sigma_n$: $$p^{(0)}_I=\prod_{i\in [n]\setminus
    I}a_{i0}\prod_{i\in I}a_{i1},\quad p^{(1)}_I=\prod_{i\in
    [n]\setminus I}b_{i0}\prod_{i\in I}b_{i1}\qquad\mbox{for all
  }I\subseteq[n].$$


 Denote by $\mathbf{A}$ the affine subspace given by $x_\emptyset=1$. The affine variety ${\rm Sec}(\Sigma_n)\cap \mathbf{A}$ is parametrized by
\begin{equation*}\label{eq:paramnclaw}
x_I\,\,=\,\,(1-s_1)\prod_{i\in I} a_{i1}+s_1\prod_{i\in I} b_{i1}.
\end{equation*}

Consider a sequence of two $\mathcal L$-cumulant transformations. The first one corresponds to the
lattice $\mathcal{L}_1$ of all \emph{one-cluster partitions} of $[n]$, i.e.\
partitions with at most one block of size greater than one. The second
one comes from the lattice $\mathcal{L}_2$ of interval partitions of
$[n]$.  The first map $\psi_1\colon \mathbf{A}\rightarrow \mathbf{A}$ is
defined by
\begin{equation*}\label{eq:cremonaxtoy}
y_I=\sum_{A\subseteq I}(-1)^{|I\setminus A|}x_A\prod_{i\in I\setminus A}x_i, \quad\mbox{for }I\subseteq [n].
\end{equation*}
The second map $\psi_2\colon\mathbf{A}\rightarrow \mathbf{A}$ is given
by~\ref{eq:cremonagen}, i.e.
\begin{equation*}\label{eq:cremonaytoz}
z_I=\sum_{\pi\in \mathcal{I}(I)}(-1)^{|\pi|-1}\prod_{B\in \pi}y_B,\quad\mbox{for }I\subseteq [n].
\end{equation*}

To see how this sequence of maps can be written as a single $\mathcal
L$-cumulant transformation we refer to \cite{pwz-2010-cumulants}. By \cite[Lemma 3.1]{MOZ}, for every $I\subseteq [n]$ such that $|I|\geq 2$, the result of $\psi_2\circ \psi_1$ applied to $\Sec(\Sigma_n)$ is 
\begin{equation}
z_I\,\,\,=\,\,\,s_1(1-s_1)(1-2s_1)^{|I|-2}\prod_{i\in I} (b_{i1}-a_{i1}).\label{eq:6}
\end{equation}
Taking $d_i=(1-2s_1)(b_{i1}-a_{i1})$ for $i\in [n]$, and
$t=s_1(1-s_1)(1-2s_1)^{-2}$ in~\eqref{eq:6} we conclude
that that secant variety, when expressed in cumulants, becomes locally
toric with $z_I=t\prod_{i\in I}d_i$ for $|I|\geq 2$. \end{example}
\noindent 
This simple local description of  $\Sec(\Sigma_n)$ can be generalized to the secant variety of the Segre product of projective spaces of arbitrary dimensions. This gives  the following result:
\begin{theorem}[see \cite{MOZ}] The secant variety ${\rm Sec}(\Segre
  (r_1,\ldots,r_k))$ is covered by normal affine toric varieties. In
  particular, it has rational singularities. \end{theorem} It turns
out that similar techniques can be applied to study the tangential
variety $T(\Segre (r_1,\ldots,r_k))$ (see \cite{MOZ} for details).

\bigskip

\noindent {\bf Acknowledgments:}{ We thank Mateusz Micha\l{}ek for his helpful comments on an early version of the paper. M. A. Cueto was partially supported
  by an AXA Mittag-Leffler postdoctoral fellowship (Sweden), by an
  Alexander von Humboldt Postdoctoral Research Fellowship and by an
  NSF postdoctoral fellowship DMS-1103857 (USA). C. Ciliberto and
  M. Mella have been partially supported by the Progetto PRIN
  ``Geometria sulle variet\`a algebriche'' MIUR.  P. Zwiernik was
  partially supported by an AXA Mittag-Leffler postdoctoral fellowship
  (Sweden), by Jan Draisma's Vidi grant from the Netherlands
  Organisation for Scientific Research (NWO) and by the European Union
  Seventh Framework Programme (FP7/2007-2013) under grant agreement
  PIOF-GA-2011-300975. This project started at the Institut
  Mittag-Leffler during the Spring 2011 program on ``Algebraic
  Geometry with a View Towards Applications.''  We thank IML for its
  wonderful hospitality.  }

\end{document}